\documentclass[12pt]{amsart}
 
 \usepackage[english, french]{babel}
 
\usepackage{amsmath}
\usepackage{amssymb}
\usepackage{caption}
\usepackage{subcaption}
\usepackage{float}
\usepackage[curve]{xypic}
\usepackage{cite}
\usepackage{enumitem}
\usepackage{color}
\usepackage{url}
\usepackage[usenames,dvipsnames]{xcolor}
\usepackage{graphicx}
\usepackage{verbatim}
\usepackage{tikz}
\usepackage[normalem]{ulem}
\usepackage{hyperref}
\hypersetup{
     colorlinks   = true,
     citecolor    = black,
     linkcolor    = blue
}

\makeatletter 
\def\@cite#1#2{{\m@th\upshape\bfseries%
[{#1\if@tempswa{\m@th\upshape\mdseries, #2}\fi}]}}
\makeatother 

\theoremstyle{plain}
\newtheorem{thm}{Theorem}[section]

\newtheorem{cor}[thm]{Corollary}
\newtheorem{prop}[thm]{Proposition}
\newtheorem{lem}[thm]{Lemma}
\newtheorem{lemma}[thm]{Lemma}
\newtheorem{sublem}[thm]{Sublemma}
\theoremstyle{definition}
\newtheorem{defn}[thm]{Definition}

\newtheorem{ass}[thm]{Assumption}
\newtheorem{ex}[thm]{Example}

\newtheorem{prob}[thm]{Problem}

\theoremstyle{}
\newtheorem{rem}[thm]{Remark}

\numberwithin{equation}{subsection}
\renewcommand{\bold}[1]{\medskip \noindent {\bf #1 }\nopagebreak}

\newcommand{\nc}{\newcommand}
\newcommand{\rnc}{\renewcommand}
\newcommand{\e}{\varepsilon}

\nc\bA{\mathbb{A}}
\nc\bB{\mathbb{B}}
\nc\bC{\mathbb{C}}
\nc\bD{\mathbb{D}}
\nc\bE{\mathbb{E}}
\nc\bF{\mathbb{F}}
\nc\bG{\mathbb{G}}
\nc\bH{\mathbb{H}}
\nc\bI{\mathbb{I}}
\nc{\bJ}{\mathbb{J}} 
\nc\bK{\mathbb{K}}
\nc\bL{\mathbb{L}}
\nc\bM{\mathbb{M}}
\nc\bN{\mathbb{N}}
\nc\bO{\mathbb{O}}
\nc\bP{\mathbb{P}}
\nc\bQ{\mathbb{Q}}
\nc\bR{\mathbb{R}}
\nc\bS{\mathbb{S}}
\nc\bT{\mathbb{T}}
\nc\bU{\mathbb{U}}
\nc\bV{\mathbb{V}}
\nc\bW{\mathbb{W}}
\nc\bY{\mathbb{Y}}
\nc\bX{\mathbb{X}}
\nc\bZ{\mathbb{Z}}
\nc\cA{\mathcal{A}}
\nc\cB{\mathcal{B}}
\nc\cC{\mathcal{C}}
\rnc\cD{\mathcal{D}}
\nc\cE{\mathcal{E}}
\nc\cF{\mathcal{F}}
\nc\cG{\mathcal{G}}
\rnc\cH{\mathcal{H}}
\nc\cI{\mathcal{I}}
\nc{\cJ}{\mathcal{J}} 
\nc\cK{\mathcal{K}}
\rnc\cL{\mathcal{L}}
\nc\cM{\mathcal{M}}
\nc\cN{\mathcal{N}}
\nc\cO{\mathcal{O}}
\nc\cP{\mathcal{P}}
\nc\cQ{\mathcal{Q}}
\rnc\cR{\mathcal{R}}
\nc\cS{\mathcal{S}}
\nc\cT{\mathcal{T}}
\nc\cU{\mathcal{U}}
\nc\cV{\mathcal{V}}
\nc\cW{\mathcal{W}}
\nc\cY{\mathcal{Y}}
\nc\cX{\mathcal{X}}
\nc\cZ{\mathcal{Z}}
\nc\wt{\widetilde}

\newcommand{\bk}{{\mathbf{k}}}

\newcommand{\M}{\mathcal{M}}

\newcommand{\R}{\mathbb R}
\newcommand{\Z}{\mathbb Z}
\newcommand{\Q}{\mathbb Q}

\newcommand{\GL}{\mathrm{GL}}

\renewcommand{\Xmin}{(X_{\operatorname{min}}, \omega_{\operatorname{min}})}
\newcommand{\Qmin}{(Q_{\operatorname{min}}, q_{\operatorname{min}})}
\newcommand{\piX}{\pi_{X_{\operatorname{min}}}}
\newcommand{\piQ}{\pi_{Q_{\operatorname{min}}}}

\nc{\dmo}{\DeclareMathOperator}
\rnc{\Re}{\operatorname{Re}}
\rnc{\Im}{\operatorname{Im}}
\dmo{\rank}{rank}
\dmo{\End}{End}
\dmo{\Hom}{Hom}
\dmo{\Jac}{Jac}
\dmo{\Id}{Id}
\dmo{\Ann}{Ann}
\dmo{\Area}{Area}
\dmo{\CP}{\bC P^1}
\dmo{\Aut}{Aut}

\title{Marked points on translation surfaces}
\author[Apisa]{Paul~Apisa}
\author[Wright]{Alex~Wright}
\begin{document}
\maketitle
\thispagestyle{empty}

\selectlanguage{english}
\begin{abstract}
We show that all  $GL^+(2, \bR)$ equivariant point markings over orbit closures of translation surfaces arise from branched covering constructions and periodic points, completely classify such point markings over strata of quadratic differentials, and give applications to the finite blocking problem. 
\end{abstract}

\setcounter{tocdepth}{1} 
\tableofcontents

%

\section{Introduction}\label{S:intro}

In this paper, we give new results on the $GL^+(2, \bR)$ action on moduli spaces of translation surfaces with marked points,  and applications such as the following.

\bold{The finite blocking problem.} We say that two, not necessarily distinct, points $x_1,x_2$ on a rational polygon  are finitely blocked if there is a finite set $B$ of  points such that every billiard trajectory from $x_1$ to $x_2$ passes through a point of $B$. We call a polygon Gaussian if it can be tiled by isometric  $(\frac\pi4, \frac\pi4, \frac\pi2)$ triangles in such a way that triangles with overlapping edges  are related via reflection in that edge. Similarly we call it rectangle-tiled if it can be tiled by rectangles and Eisenstein if it can be tiled by  $(\frac\pi6, \frac\pi3, \frac\pi2)$ triangles. We adopt the convention that all polygons are required to have connected boundary and do not have slits.

\begin{thm}\label{T:poly}
Let $P$ be a rational polygon. 
\begin{enumerate}
\item If $P$ is Gaussian, Eisenstein, or rectangle-tiled, then any two points are finitely blocked. 
\item If $P$ is not Gaussian, Eisenstein, or rectangle-tiled and all its angles are multiples of $\pi/2$, then possibly infinitely many pairs of points are finitely blocked, but each point is finitely blocked from only finitely many other points. 
\item Otherwise, only finitely many pairs of points are finitely blocked in $P$.  
\end{enumerate}
\end{thm}
The main content is the third statement; the first two are included for completeness and are closely related to previous results (see for example Theorems 1 and 2 of Leli\`evre, Monteil, Weiss~\cite{LMW}). We prove Theorem \ref{T:poly} by showing that every translation surface that is not a branched cover of a torus has a finite set that, together with covering maps to half-translation surfaces, accounts for all finite blocking (see Theorems \ref{T:cor} and \ref{T:blockingsets}). Theorem~\ref{T:poly} builds upon and recovers results of authors such as Gutkin, Hubert, Leli\`evre, Monteil, Schmidt, Schmoll, Troubetzkoy, and Weiss, which we will discuss shortly. Our results also apply to the illumination problem, which is the special case of the finite blocking problem when the blocking set is required to be empty.

\bold{Affine invariant submanifolds.} Given a partition of $2g-2$ as a sum of positive integers $2g-2=\sum_{i=1}^s k_i$, define the stratum $\cH(k_1, \ldots, k_s)$ to be the orbifold of all translation surfaces $(X,\omega)$ where $X$ has genus $g$ and $\omega$ has zeros of order $k_1, \ldots, k_{s}$. A result of Eskin-Mirzakhani-Mohammadi \cite{EM, EMM} gives that any closed $GL^+(2, \bR)$ invariant subset of a stratum is an affine invariant submanifold, which is by definition a properly immersed smooth orbifold whose image is locally described by real homogeneous linear equations in period coordinates. 

\bold{Marked points on translation surfaces.} Let $\cM$ be an affine invariant submanifold of a stratum $\cH$ of  translation surfaces. Let $\cH^{*n}$ denote the set of surfaces in $\cH$ with $n$ distinct marked points, none of which coincide with each other or with zeros of the Abelian differential. Let $\pi: \cH^{*n}\to \cH$ be the map that forgets the marked points.  

Define an $n$-point marking over $\cM$ to be an affine invariant submanifold $\cN$ of $\cH^{*n}$ such that $\pi(\cN)$ is equal to a dense subset of $\cM$ (equivalently, $\pi(\cN)$ contains $\cM$ minus a finite, possibly empty, union of smaller dimensional affine invariant submanifolds). Define an $\cM$-periodic point to be a $1$-point marking over $\cM$ of the same dimension as $\cM$. 

\begin{thm}[Eskin-Filip-Wright]\label{T:periodic}
An affine invariant submanifold has infinitely many periodic points if and only if it consists entirely of branched covers of tori.
\end{thm}

 Section \ref{SS:EFW} explains why  Theorem \ref{T:periodic} is a special case of results in \cite{EFW}. 
 
If $\cM$ is a closed orbit containing $(X,\omega)$, the $\cM$-periodic points correspond to finite orbits of the affine symmetry group on $(X,\omega)$.  As we will discuss, in this case Theorem \ref{T:periodic} was previously known \cite{GHS,M2}.

An easy example of periodic points is provided by Weierstrass points whenever $\cM$ consists of hyperelliptic surfaces. 

Given an $n'$-point marking $\cN'$ over $\cM$, and an $n''$-point marking $\cN''$ over $\cM$, we will say that any irreducible component of 
 $$\{(X,\omega, S'\cup S''): (X, \omega, S')\in \cN', (X, \omega, S'')\in \cN'', |S_1\cup S_2|=n'+n''\}$$
 is a fiberwise union of $\cN'$ and $\cN''$. One defines a fiberwise union of more than two point markings in the same way. 
 
 We say that an $n$-point marking $\cN$ over $\cM$ is \emph{reducible} if it is a fiberwise union of other point markings.  We call a point marking irreducible if it is not reducible.  
 
 Every point marking is a fiberwise union of irreducible point markings, which we call its irreducible pieces. 
 In this way,
%
%
 %
 %
 the study of point markings immediately reduces to the irreducible case.

The following result  states that when $\cM$ does not consist entirely of torus covers the only non-obvious ways to mark points over $\cM$ are to mark $\cM$-periodic points. This result arose during conversations with Ronen Mukamel.  

\begin{thm}\label{T:main}
Let $\cM$ be an affine invariant submanifold that does not consist entirely of branched covers of tori. Any irreducible $n$-point marking $\cN$ over $\cM$ with $n>1$ arises from a half-translation surface covering construction: for any $(X, \omega, S)\in \cN$, there is a translation covering map from $(X,\omega)$ to a half-translation surface that takes $S$ to a single point. 
\end{thm} 

A translation covering map from a translation surface $(X,\omega)$ to a quadratic differential $(Q,q)$ is defined to be a branched covering $f: X \to Q$ of Riemann surfaces such that $\omega^2 = f^*(q)$. A Riemann surface with a non-zero quadratic differential is also called a half-translation surface.

Theorem \ref{T:main} implies a stronger statement, which we allude to with the terminology ``covering construction". See Remark \ref{R:main}.
 In the case that $\cM$ does consist entirely of branched covers of tori, point markings are easily described, and there are infinitely many  irreducible $n$-point markings for all  $n\geq 1$.

\bold{Strata of quadratic differentials.} For any  connected component of a stratum $\cQ$ of quadratic differentials, one can form the affine invariant submanifold $\tilde{\cQ}$ consisting of all Abelian differentials which arise as double covers (also called square roots) of quadratic differentials in $\cQ$. Each $(X,\omega)\in \tilde{\cQ}$ has a natural involution $J$, so that $(X/J, \omega^2)\in \cQ$. (Since $(X,\omega)$ might, in unusual cases, have more than one involution, the data of $J$ should be included in a point of $\tilde\cQ$, but rather than write $(X,\omega, J)\in \tilde\cQ$ we suppress this from the notation.)

If $\cQ$ consists entirely of hyperelliptic  surfaces we say that $\cQ$ is hyperelliptic. In this case the hyperelliptic involution on $X/J$ lifts to a hyperelliptic involution on $X$.

\begin{thm}\label{T:Q}
Suppose $\cQ$ is not $\cQ(-1^4)$, $\cQ(2,-1^2)$, or $\cQ(2^2)$. 
If $\cQ$ is not hyperelliptic, the only $\tilde{\cQ}$-periodic points are fixed points of the involution $J$. If $\cQ$ is hyperelliptic, the only $\tilde{\cQ}$-periodic points are fixed points for $J$ and fixed points for the hyperelliptic involution. 
\end{thm}

In Section \ref{S:background} we will recall the definition of rank, and in Corollary \ref{C:q-rank-one} we will see that $\cQ(-1^4)$, $\cQ(2,-1^2)$, and $\cQ(2^2)$ are exactly the $\cQ$ that have rank 1. 

Theorem \ref{T:Q} is the most difficult of our main theorems, and we expect the combination of Theorems \ref{T:main} and \ref{T:Q} to have the most applications.   

\bold{Context.} Leli\`evre, Monteil, and Weiss recently showed that the only translation surfaces in which every pair of points are finitely blocked are covers of tori \cite{LMW}; we recover and strengthen this in Corollary~\ref{C:LMW}. Generalizing work of  Hubert, Schmoll, and Troubetzkoy \cite{HST} in the case of lattice surfaces, they also showed that for any point $p$ on any translation surface, the set of points not illuminated by $p$ is finite. For more background on the finite blocking problem and the closely related illumination problem, see \cite{Mont1, Mont2, HST, LMW}.  


In the case of closed $GL^+(2, \bR)$ orbits, Theorem \ref{T:periodic} is due to Gutkin, Hubert, and Schmidt \cite{GHS} and was  established independently by M\"oller using algebro-geometric methods \cite{M2}. M\"oller also showed that for non-arithmetic closed orbits   in genus 2 the only periodic points are Weierstrass points. Non-arithmetic closed orbits in genus 2 were classified by McMullen \cite{McM:spin, Mc4}.  

The proof of Theorem \ref{T:main} builds upon and was inspired by an argument of Hubert, Schmoll, and Troubetzkoy \cite[Theorem 5]{HST}.

Lanneau classified connected components of strata of quadratic differentials \cite{Lconn}, see also \cite{CM} for a correction. 

Apisa previously classified periodic points over connected components of strata of Abelian differentials: they exist only for hyperelliptic connected components, in which case they must be Weierstrass points \cite{Apisa}. We make use of one of the main technical lemmas of \cite{Apisa} in Section \ref{S:FindCyl}. As we point out in Remark \ref{R:QvsAb}, the strategy used to prove Theorem \ref{T:Q} does not work for strata of Abelian differentials, so our analysis does not recover Apisa's result. 

There is an unexpected periodic point over the golden eigenform locus in genus 2, see forthcoming work of Eskin-McMullen-Mukamel-Wright and \cite{KM2}. 

\begin{prob}
Compute the periodic points over the (Prym) eigenform loci and the Veech-Ward-Bouw-M\"oller Teichm\"uller curves, as well as the new orbit closures in the forthcoming work of Eskin-McMullen-Mukamel-Wright and \cite{MMW}. 
\end{prob}

When studying translation surfaces without marked points it is often helpful to consider degenerations which may have marked points \cite{MirWri}, and indeed Theorems \ref{T:main} and \ref{T:Q} have already been used to study $GL^+(2, \bR)$ orbit closures of unmarked translation surfaces \cite{MirWri2}.

\bold{Organization.} Section \ref{S:Tmain} proves Theorem \ref{T:main}. Section \ref{S:FB} gives our applications to the finite blocking problem, including Theorem \ref{T:poly}. The remaining sections, which are independent of Section \ref{S:FB}, prove Theorem \ref{T:Q}. Section \ref{S:background} gives the required background, and Section \ref{S:Q-overview} gives a proof, conditional on two results established in the remaining two sections.  The approach is by induction: Section \ref{S:FindCyl} produces an appropriate cylinder to degenerate, and Section \ref{S:BaseCase} provides the base case.

\bold{Acknowledgments.} The authors thank Corentin Boissy, Elise Goujard, Erwan Lanneau, Maryam Mirzakhani,  Barak Weiss, Christopher Zhang, and Anton Zorich for helpful conversations, and Ronen Mukamel for significant contributions to this paper regarding Theorem \ref{T:main}. The authors are also grateful to the referees for helpful comments. This research was partially conducted during the period AW  served as a Clay Research Fellow. This material is based upon work supported by the National Science Foundation Graduate Research Fellowship Program under Grant No. DGE-1144082. This material is also based upon work supported by the National Science Foundation under Award No. 1803625. PA gratefully acknowledges their support.

\section{Proof of Theorem \ref{T:main}}\label{S:Tmain}


We begin with some preliminary observations. 

\begin{lem}\label{L:FiberDim}
Let $\cN$ be a point marking over $\cM$, and let $(X, \omega, S)\in \cN$. For each $s\in S$, fix a path $\gamma_s$ from a zero of $\omega$ to $s$. Then each branch of the fiber of $\cN$ over $(X,\omega)$ is locally defined by a collection of linear equations of the form 
$$\sum_{s\in S} a_s \int_{\gamma_s}\omega = \int_\gamma \omega,$$
where $a_s\in \bR$, and $\gamma\in H_1(X, \Sigma, \bR)$, and $\Sigma$ is the set of zeros of $\omega$. 

Moreover, these equations can be chosen to hold locally on $\cN$, and each fiber of $\cN$ over $\cM$ is either empty or has dimension $\dim \cN-\dim \cM$. 
\end{lem}

 Note that $\int_\gamma \omega$ is locally constant on the fiber of $\cN$ over $(X,\omega)$, so in the first part of the statement we could have replaced it with a constant.

\begin{proof}
Like all $GL^+(2, \bR)$ orbit closures, $\cN$ is locally defined by linear equations in local period coordinates. A basis for local period coordinates can be obtained from the $\gamma_s$ together with a basis of $H_1(X, \Sigma)$.   

Since the equations hold locally on $\cN$, the dimension of $\cN$ must be equal to the dimension of $\cM$ plus the dimension of a fiber; the dimension of the fiber is $|S|$ minus the dimension of the span of the coefficient vectors $(a_s)_{s\in S}$ arising from the linear equations locally defining the fiber. 
\end{proof}

 Note that the image of $\cN$ in $\cM$ is $GL^+(2,\bR)$ invariant and contains an open subset of $\cM$, so by \cite{EM, EMM} this image contains the complement of a union of smaller dimensional affine invariant sub-manifolds.  However, as we now show, some fibers can indeed be empty.

\begin{ex}
Let $\cM\subset \cH(5,1)$ be the locus of double covers of surfaces in $\cH(2)$ that are branched over the zero and one other point $p$. Let $\cN$ be the 2-point marking over $\cM$ that marks the two pre-images of the image of $p$ under the hyperelliptic involution. Then  the image of $\cN$ in $\cM$ is the locus of covers where $p$ is not a Weierstrass point. 
\end{ex}


We now make one final preliminary remark before moving on to the main definition in this section. 
Let $\cN$ be an irreducible 2-point marking over $\cM$, and let $(X, \omega, \{p_1, p_2\}) \in \cN$.

If $\dim \cN=\dim \cM$, then $\cN$ would be the fiberwise union of two periodic points, contradicting the fact that $\cN$ is irreducible. On the other hand, if $\dim \cN=\dim \cM+2$, then $\cN$ would be the fiberwise union of two  1-point markings in which the single point is unconstrained, again contradicting the fact that $\cN$ is irreducible. Hence, we get that $\dim \cN=\dim \cM+1$. 

Let $\gamma_i$ be a path from a zero of $\omega$ to $p_i$. By Lemma \ref{L:FiberDim}, the fiber of $\cN$ over $(X,\omega)$ is locally defined by a single equation of the form 
$$a_1\int_{\gamma_1}\omega + a_2\int_{\gamma_1} \omega   = \int_\gamma \omega,$$
where $\gamma\in H_1(X, \Sigma, \bR)$ and $a_1, a_2\in \bR$. If $a_1$ was zero, then $\cN$ would be the fiberwise union of 1-point marking containing $(X,\omega, \{p_1\})$, in which the single point is unconstrained, and a periodic point containing $(X,\omega, \{p_2\})$, yet again contradicting the fact that $\cN$ is irreducible. So we conclude that both $a_i$ are non-zero.

\begin{defn}\label{D:slope}
Let $\cN$ be an irreducible 2-point marking over $\cM$. As above, consider any $(X,\omega, \{p_1,p_2\})\in \cN$, and rewrite the linear equation locally  defining the fiber in the form 
$$\int_{\gamma_1} \omega = a\int_{\gamma_2}   \omega + \int_\gamma \omega,$$
with $a\neq 0$.
We define the slope of $\cN$ to be $a$ or $1/a$, whichever is larger in absolute value. 
\end{defn}

The slope describes the speed at which one marked point moves when the other marked point is moved at unit speed  (and the underlying  surface $(X,\omega)$ without marked points is fixed). Note that if the role of $p_1$ and $p_2$ are interchanged, $1/a$ will play the role of $a$. 

The slope depends only on $\cN$ and not on any of the choices made in the definition. Indeed, over $\cN$ there is a bundle whose fiber is first cohomology rel the singular points and the marked points. There is a rank two trivial subbundle spanned by the  deformations that move the two marked points. (We are assuming the two marked points are labeled; if not, one can pass to a two-fold cover where they are labeled.) We can intersect $T(\cN)$ with this rank two trivial subbundle to get a trivial subbundle, and $a$ is the slope of that subbundle in each fiber. 

\begin{rem}\label{R:MorePoints}
Let $\cN$ be an irreducible $n$-point marking over $\cM$ of dimension $\dim \cM+1$. Then if $(X, \omega, S) \in \cN$ and $p_1,p_2 \in S$, one can define the slope of $p_1$ with respect to $p_2$ in the same way. If $(X, \omega, S)$ is generic, this will be equal to the slope of the irreducible 2-point marking given by the orbit closure of $(X, \omega, \{p_1,p_2\})$. 
\end{rem}

\begin{ex}
If $\cM$ is the hyperelliptic locus in some stratum, and $\cN$ is given by a pair of points that are exchanged under the hyperelliptic involution, then the slope is -1. 
\end{ex}

\begin{ex}
Suppose that $\cM$ is an affine invariant submanifold of translation surfaces of genus $g$, such that the generic translation surface in $\cM$ has a unique  translation covering  map to a translation surface of genus $h<g$. Let $\cN$ be the set of all pairs of points on surfaces in $\cM$ that map to the same point on the associated surface of genus $h$. Then $\cN$ has slope 1. 
\end{ex}

\begin{ex}
Suppose that $\cM$ is an arithmetic Teichm\"uller curve, so each surface $(X,\omega)\in \cM$ admits a unique translation covering map $f$ of degree $d$ to a torus branched over one point. Let $\cN$ be  the locus of all pairs of points $p_1,p_2$ on surfaces $(X,\omega)\in \cM$ such that $2f(p_1)= 7f(p_2)$, where we view the image torus as a group with origin equal to the image of the zeros of $\omega$. Then $\cN$ has slope $7/2$. 
\end{ex}

\begin{thm}\label{T:Slope1}
Suppose that $\cM$ has only finitely many $\cM$-periodic points.  Then any irreducible $2$-point marking $\cN$ over $\cM$ has slope $1$ or $-1$. 
\end{thm} 

%
%
%



%
%
%
%
%
%
%

The proof is a generalization of  \cite[Proof of Theorem 5]{HST}. We will say that a point $p$ on a surface $(X,\omega)$ is a periodic point if it is contained in a $\cM$-periodic point, where $\cM$ is the orbit closure of $(X,\omega)$.

\begin{lem}\label{L:boundaryperiodic}
Let $\cN$ be an  irreducible $2$-point marking over an affine invariant submanifold $\cM$. Suppose that $(X,\omega)$ has dense orbit in $\cM$. Then if $(X,\omega, \{p_1,p_2\})\in \cN$ and $p_1$ is $\cM$-periodic, then  $p_2$ is also $\cM$-periodic.

Furthermore, if $(X,\omega,\{p_1,p_2\})$ is  in the closure of the fiber of $\cN$ over $(X,\omega)$ and  $p_1$ is  a zero of $\omega$, then $p_2$ is either a zero of $\omega$ or an $\cM$-periodic point.

Similarly, if $(X,\omega,\{p_1,p_2\})$ is  in the closure of the fiber of $\cN$ over $(X,\omega)$ and $p_1=p_2$, then $p_1$ is a $\cM$-periodic point or a zero.
\end{lem}

As in Remark \ref{R:MorePoints}, Lemma \ref{L:boundaryperiodic} can also be applied to irreducible $n$-point markings $\cN$ with $\dim \cN =\dim \cM+1$, by forgetting all but two points.

\begin{proof} 
Let $\cN'$ be the orbit closure of $(X,\omega, \{p_1,p_2\})$, where $p_1$ is $\cM$-periodic. This affine invariant submanifold must be properly contained in $\cN$, because at the generic point of $\cN$ neither of the marked points are periodic. 

By the discussion before Definition \ref{D:slope},  $\dim \cN = \dim \cM +1$. Since $\cN' \subsetneq \cN$, we must have $\dim \cN'=\dim \cM$, and hence both $p_1$ and $p_2$ are $\cM$-periodic. 

Now, suppose to a contradiction that $p_1$ is a zero of $\omega$ and $p_2$ is neither a zero nor an $\cM$-periodic point. Then the $GL^+(2,\bR)$ orbit closure of $(X,\omega, p_1, p_2)$ has dimension $\dim \cM+1 = \dim \cN$ and contains the set of $(X,\omega, p_1, p_2')$ where $p_2'$ is arbitrary. 

We can now use either general principles or concrete arguments. The general principle is that the  orbit closure of $(X,\omega, p_1, p_2)$ is contained in the boundary of $\cN$ (in the Hodge bundle over $\cM_{g,2}$), and by \cite{MirWri} or \cite{Fi1} this boundary must have dimension strictly smaller than $\dim \cN$. One concrete argument, which we only sketch, is that the fiber of $\cN$ over $(X,\omega)$ consists of linear submanifolds of  $(X,\omega)\times (X,\omega)$ of slope $a$, and such submanifolds cannot accumulate on the set of  $(X,\omega, p_1, p_2')$, which has infinite slope. 

The proof of the final claim is almost identical. 
\end{proof}

\begin{proof}[Proof of Theorem \ref{T:Slope1}.]
Suppose otherwise.
Consider  $(X,\omega, \{p_1,p_2\})\in \cN$ such that the $GL^+(2, \bR)$ orbit of $(X,\omega)$ is dense in $\cM$. Let $\Sigma'$ be the finite set of all zeros of $\omega$ together with all points of $(X,\omega)$ contained in $\M$-periodic points. By assumption $\Sigma'$ is finite.

We consider the fiber of $\cN$ over $(X,\omega)$; by the discussion in Definition \ref{D:slope}, it must be one dimensional.  Any path that pushes the point $p_1$ around $(X, \omega)$ while avoiding the set $\Sigma'$ can be lifted to a path in the fiber that pushes $p_1$ and $p_2$ around $(X, \omega)$.  By Lemma \ref{L:boundaryperiodic}, along any such path where $p_1$ avoids $\Sigma'$, the point $p_2$ does not collide with either $p_1$ or $\Sigma'$. 

As we move $p_1$ on $(X,\omega)$, $p_2$ moves in a  way determined by the slope so that the marked surface remains in the fiber. After possibly swapping the role of $p_1$ and $p_2$, we may assume that when $p_1$ moves with unit speed the point $p_2$ moves with speed strictly less than unit speed. 

By Lemma \ref{L:boundaryperiodic}, when $p_1$ is moved to a point of $\Sigma'$, it must be the case that $p_2$ is also moved to a point of $\Sigma'$.
Let $\alpha$ be a saddle connection on $(X,\omega, \Sigma')$ (that is, a straight line segment from a point of $\Sigma'$ to a point of $\Sigma'$ whose interior is disjoint from $\Sigma'$) of minimal length. Move $p_1$ to be on $\alpha$, and then move $p_1$ from one end of $\alpha$ to the other. At either end, $p_1$ is at a point of $\Sigma'$, so $p_2$ must also be at a point of $\Sigma'$. This is a contradiction because $p_1$ moves strictly faster than $p_2$ and because $\alpha$ is a minimal length saddle connection.  
\end{proof}

\begin{lem}\label{L:producecover}
Suppose that $\cM$ has only finitely many $\cM$-periodic points.
Let $\cN$ be an irreducible $n$-point marking over $\cM$ with $\dim \cN=\dim \cM+1$. For every $(X,\omega, S)\in \cN$ there is a translation covering map $f$ to a half-translation surface such that $f$ maps $S$ to a point. 

If the slope of $\cN$ is 1, then $f$ can be chosen to map to a translation (rather than half-translation) surface.
\end{lem}

\begin{proof}
We first prove the main claim.

It suffices to prove the lemma when $(X,\omega, S)$ has dense orbit in $\cN$, so we assume this. The reason it is sufficient is because if the conclusion holds on a dense subset of $\cN$, then it holds on all of $\cN$, and a dense subset of $\cN$ has dense orbit. 

Suppose, for some $n'>n$, that there is an irreducible $n'$-point marking $\cN'$ over $\cM$ of dimension $\dim \cN=\dim \cM +1$, such that for all $(Y, \eta, P')$ in $ \cN'$ minus a union of smaller affine invariant submanifolds  there is a set $P \subset P'$ so that $(Y, \eta, P)\in \cN$. 

In other words, in $\cN$ there are $n$ marked points whose position locally determine each other given the unmarked translation surface, and $\cN'$ extends these $n$ points to a larger collection of points such that each locally determines all the others. 

Let $\Sigma'$ be the set of zeros of $\omega$ and all points $s$ such that $(X,\omega, s)$ is contained in an $\cM$-periodic point. 

\begin{sublem}\label{SL:N'Bound}
Let $T$ be the sum of the cone angles at points of $\Sigma'$, divided by $\pi$. Then $n' \leq T$.
\end{sublem}
\begin{proof}
Consider $(X,\omega, S')\in \cN'$, and move the $n'$-marked points around while remaining in the fiber of $\cN'$ over $(X,\omega)$. Move one of the marked points along a horizontal separatrix until it hits a point of $\Sigma'$. By Lemma \ref{L:boundaryperiodic}, all marked points must then lie at points of $\Sigma'$.   

By Theorem \ref{T:Slope1}, the slope for any pair of these points is $1$ or $-1$, so each marked point must have traveled along a different directed horizontal line segment towards a point of $\Sigma'$. There are exactly $T$ such directed horizontal line segments. The claim is proved. 
\end{proof}

Now assume $n'$ as above was maximal, which is possible by Sublemma \ref{SL:N'Bound}. Let $\cN'$ be a corresponding $n'$-point marking.

\begin{sublem}\label{SL:EC}
Suppose that $p$ and $q$ are two points in $(X, \omega) - \Sigma'$, possibly equal. Say that $p \sim q$ if there is a collection $S$ of $n'$ points containing $p$ and $q$ and such that $(X, \omega; S) \in \cN'$. Then $\sim$ is an equivalence relation on $(X, \omega) - \Sigma'$.
\end{sublem}


\begin{proof}[Proof of Sublemma \ref{SL:EC}:]
Recall that Lemma \ref{L:boundaryperiodic} gives that if $(X, \omega; S) \in \cN'$ and one point of $S$ is in $\Sigma'$, then all points of $S$ are in $\Sigma'$. 

We begin with a subtle point, which is that $p\sim p$ for all $p\in (X, \omega) - \Sigma'$, for which we need to prove that for all $p$, we can find $S$ containing $p$ so that $(X, \omega; S) \in \cN'$. As discussed previously, in the fiber over $(X,\omega)$, the position of one of the $n'$ points locally determines the position of the other $n'-1$ points. Lemma \ref{L:boundaryperiodic} implies that as we move one point in $(X,\omega)\setminus \Sigma'$, the other points stay in $(X,\omega)\setminus \Sigma'$ and that none of the $n'$ points collide with each other. So, we can see that $p\sim p$ by starting with any $(X, \omega; S) \in \cN'$, and moving $S$ so that it contains $p$. 

If $\sim$ is not an equivalence relation, then there are two sets of $n'$-points, both in $\cN'$, that partially overlap. Let $S''$ denote their union, and let $\cN''$ denote the orbit closure of $(X,\omega,S'')$. Since we have assumed $(X,\omega)$ has dense orbit, $\cN''$ is a point marking over $\cM$. 

We now show that $\dim \cN'' = \dim \cM+1$ and that $\cN''$ is irreducible. Indeed, $\cN''$ is a component of 
$$\{(X,\omega, S\cup S'): (X,\omega, S)\in \cN', (X,\omega, S')\in \cN', |S \cup S'|=|S''|\}.$$
In fibers of  $\cN'$, the position of each point locally determines the position of all the others.
Since $S \cap S' \neq \emptyset$, we get the same statement for fibers of $\cN''$, proving $\dim \cN'' = \dim \cM+1$ and that $\cN''$ is irreducible. This contradicts the maximality of $n'$.
\end{proof}

Let $Q_0$ denote the quotient of $(X,\omega)\setminus \Sigma'$ by the equivalence relation $\sim$. We claim that the map $(X,\omega)\setminus \Sigma'\to Q_0$ is a covering map. Indeed, suppose that $(X,\omega, \{z_1,\cdots, z_{n'}\})\in \cN'$,  and that none of the $z_i$ are in $\Sigma'$.  For some $\e>0$, the $\e$ balls about the $z_{i}$ are embedded, disjoint, and disjoint from $\Sigma'$. Since each marked point moves at speed $1$ or $-1$ with respect to each other marked point, each point in the ball $B_\e(z_i)$ is equivalent to a point in $B_\e(z_j)$ for each $j\neq i$. The quotient by the equivalence relation identified all these balls to a single ball. This shows that the map is a covering map, and hence that the quotient is a surface.

Furthermore, the composition of the map from $B_\e(z_i)$ to $Q_0$ with the inverse of the map from $B_\e(z_j)$ to $Q_0$ is given in local coordinates by multiplication by $1$ or $-1$. Hence the quotient has an atlas of charts to $\bC$ whose transition functions are translations and translations composed with multiplication by $-1$. The quotient map is a local isometry, so the quotient is a punctured surface. 

The quotient map extends continuously to a map $f$ from $(X,\omega)$ to the metric completion of the punctured surface $Q_0$, and this $f$ is the desired map.  

If the slope of $\cN$ is 1, a slight variant of the above gives the final claim, modifying the proof by picking $\cN'$ to be a maximal point marking containing $\cN$ subject to the condition that any pair of points in $\cN'$ has slope 1.
\end{proof}

\begin{lem}\label{L:dimplus1}
Suppose that $\cM$ has only finitely many $\cM$-periodic points. Let $\cN$ be an irreducible $n$-point marking with $n>2$. Then $\dim \cN=\dim \cM+1$. 
\end{lem}

\begin{proof}
Suppose the lemma is not true. Consider a counterexample with $n$ as small as possible.

Pick $(X,\omega, S) \in \cN$ such that the $GL^+(2,\bR)$ orbit of $(X,\omega, S)$ is dense in $\cN$. 

Because of the minimality of $n$, no point of $S$ locally determines the position of any of the others. Hence the same remains true if we forget a point $p$ in $S$; so no irreducible piece $\cN'$ of the orbit closure of $(X,\omega, S-\{p\})$ satisfies $\dim \cN'=\dim \cM+1$, unless it is given by a single unconstrained marked point.  Hence, again because of the minimality assumption, the orbit closure of $(X,\omega, S-\{p\})$ is given by $n-1$ unconstrained points. We conclude that the fiber of $(X,\omega)$ is locally defined by a single equation on the points of $S$. By irreducibly this equation must involve all the points. 

If, for each $p\in S$, we let $\gamma_p$ be a path from a zero of $\omega$ to $p$, this equation can be written as $$\sum a_i \int_{\gamma_{p_i}}\omega= \int_\gamma \omega$$ for some $\gamma \in H^1(X, \Sigma, \bR)$ and some non-zero real numbers $a_i$. 

We will use the following to finish the proof.

\begin{sublem}\label{SL:collide}
If $|S|>3$, or if $|S|=3$ and $a_2+a_3\neq 0$, then it is possible to move in the fiber of $\cN$ so that $p_1$ collides with a zero, but the other points of $S$ do not collide with each other or with zeros. 

If $|S|=3$ and $a_1+a_2\neq 0$ and $a_1+a_2+a_3\neq 0$, then it is possible to move in the fiber of $\cN$ so that $p_1$ and $p_2$ collide, but that $p_3$ does not collide with $p_1, p_2$ or a zero. 
\end{sublem}

It seems likely that Sublemma \ref{SL:collide} could be justified or avoided in different ways; we give one justification now. Note that, if, for example, $a_2+a_3=0$, then $p_2$ and $p_3$ might collide and become a single unconstrained point. 

\begin{proof}
We first prove the first statement. For convenience, rotate the surface to assume without loss of generality that  $p_1$ lies on a vertical line segment starting at a point of $\Sigma$ (also known as a separatrix).  As we now show, we can then make a perturbation so that $p_1$ is the only point of $S$ that lies on vertical line segment  starting at a point of $\Sigma$, while remaining in the fiber of $\cN$. Indeed, in local coordinates, this deformation needs to move each $p_i, i>1$ off of a countable collection of vertical lines. (These arise from finitely many vertical separatrices, which wind around the surface and may intersect a small piece of the surface countably many times.) Because we know the fiber is locally defined by one linear equation as above, we can see that, even while fixing the position of $p_1$, each $p_i$ can be moved off any vertical line segment. So in fact almost every deformation of $S$ fixing $p_1$ and staying in the fiber of $\cN$ will work. 

We now claim that it is possible to choose the deformation so that moreover  afterwards no two points of $S$ are joined by a vertical line segment. If this is not possible, there must be two points, say $p_2$ and $p_3$, so that $p_3$ remains on the vertical line through $p_2$ in all deformations fixing $p_1$. This is only possible if $n=3$ and $a_3=-a_2$. 

Now, we can move $S$ while keeping constant the real parts of the periods of all the $\gamma_i$, so that $p_1$ moves along a vertical line segment to a point of $\Sigma$. Since each point of $S$ travels along a vertical line, this does not cause any of the other points of $S$ to collide with each other or with $\Sigma$. This proves the first statement. 

We now prove the second statement. Perturbing the marked points, we can assume there is a line segment between $p_1$ and $p_2$ not parallel to any other line segment joining points of $\Sigma \cup \{p_1, p_2\}$; without loss of generality this segment is vertical.

First suppose that, no matter how we do this, $p_3$ lies on a vertical line segment starting at a point of $\Sigma$. Then, as we move $p_1$ and $p_2$ equally to the right, $p_3$ must be unchanged, so we conclude that $a_1+a_2=0$, a contradiction. 

Next suppose that, no matter how we do this, that $p_3$ lies on the continuation of the vertical line segment containing $p_1$ and $p_2$. Similarly this contradicts our assumption that $a_1+a_2+a_3\neq 0$. 

So we can assume that $p_3$ does not lie on the continuation of the vertical line segment joining $p_1$ and $p_2$ or on a vertical line segment starting at a point of $\Sigma$. Now we can use a deformation that fixes real periods to collide $p_1$ and $p_2$ without colliding $p_3$ with $p_1, p_2$ or $\Sigma$.
\end{proof}

If $|S|>3$, use Sublemma \ref{SL:collide} to move one of the points of $S$ to a point of $\Sigma$. This gives a new point marking that, by \cite{MirWri}, is defined by a single equation and contradicts the minimality of $n$. Hence $n=3$. 

Now that we know $n=3$, first suppose that $|a_2|\neq |a_3|$. Use Sublemma \ref{SL:collide} to move $p_1$ into a point of $\Sigma$. We can assume that $\gamma_{p_1}$ was a path from the zero that $p_1$ collides with to $p_1$, so that $\gamma_{p_1}$ becomes a constant path after this collision. 
We get a two-point marking defined by 
$$a_2\int_{\gamma_{p_2}}\omega + a_3 \int_{\gamma_{p_3}}\omega = \int_\gamma\omega.$$ 
This contradicts Theorem \ref{T:Slope1}, because it is an irreducible 2-point marking of slope not equal to $\pm1$. 

Next suppose that $|a_1|=|a_2|=|a_3|$. We can assume that $a_1$ and $a_2$ have the same sign, so $a_1+a_2\neq 0$ and $|a_1+a_2|\neq |a_3|$. Use Sublemma \ref{SL:collide} to collide $p_1$ and $p_2$. This gives a two-point marking defined by $$(a_1+a_2)\int_{\gamma_{p_1}}\omega + a_3 \int_{\gamma_{p_3}}\omega = \int_{\gamma+ a_2(\gamma_{p_1}-\gamma_{p_2})} \omega ,$$
where we note that since $p_1=p_2$, the path $\gamma_{p_1}-\gamma_{p_2}$  gives a class in $H_1(X,\Sigma, \bR)$. 
Again this contradicts Theorem \ref{T:Slope1}.
\end{proof}

\begin{proof}[Proof of Theorem \ref{T:main}]
By Theorem  \ref{T:periodic}, there are only finitely many $\cM$-periodic points. By Lemma \ref{L:dimplus1}, we have $\dim \cN=\dim \cM+1$. Hence Lemma  \ref{L:producecover} gives the result. 
\end{proof}

\begin{rem}
Note that while this section has considered only translation surfaces, the results in this section hold almost verbatim for half-translation surfaces as well, except that for half-translation surfaces the slope of a point marking is only well-defined up to sign.
\end{rem}

%
%

\section{The finite blocking problem}\label{S:FB}

In the first subsection we explain the implications of Theorem \ref{T:main}  for finitely blocked points; in the next two we give applications; and in the final subsection we study possible finite blocking sets. The final three subsections can be read independently of each other but all rely on the first.  

\subsection{Consequences of Theorem \ref{T:main}}
Throughout this subsection let $(X, \omega)$ be a translation surface. Given two not necessarily distinct points $x_1$ and $x_2$ on $(X, \omega)$ the finite blocking problem asks whether all straight line paths between $x_1$ and $x_2$ may be blocked by a finite collection of points $B$. If this is possible then we say that $x_1$ and $x_2$ are blocked by $B$. The following lemma provides an example of this phenomenon.  

\begin{lem}\label{L:involution-blocking}
Suppose that $(X, \omega)$ has an involution $j$ so that $j^* \omega = - \omega$. For any point $p$ that is not a zero and is not fixed by $j$, $p$ and $j(p)$ are finitely blocked by the fixed points of $j$.
\end{lem}
\begin{proof}
Let $\ell$ be a line segment in $(X,\omega)$ joining $p$ to $j(p)$. Since $j^*\omega=-\omega$, we get that $j$ maps $\ell$ to itself, and hence contains a fixed point in its interior.
\end{proof}

 Leli\`evre, Monteil, and Weiss  showed that if $(X, \omega)$ is a translation cover of a torus then any two points are finitely blocked (and, conversely, that this property characterizes torus covers) \cite[Theorem 1]{LMW}.  Hence, in the remainder of this section, we make the following standing assumption. 
 
 \begin{ass}
 Suppose throughout this section that $(X, \omega)$ is not a translation cover of a torus.
 \end{ass}

Recall the result of M\"oller that states that,  with this assumption,  there is a unique translation covering map $\piX:(X,\omega)\to \Xmin$ to a translation surface of minimal genus, and any map from $(X,\omega)$ to a translation surface is a factor of this map \cite[Theorem 2.6]{M2}. This can be extended to quadratic differentials as follows. 

\begin{lem}\label{L:minquad}
There is a quadratic differential $\Qmin$ with a degree 1 or 2 translation covering map $\Xmin\to \Qmin$ such that any translation covering map from $(X,\omega)$ to a quadratic differential is a factor of the composite map $\piQ:(X,\omega)\to \Qmin$.
\end{lem}

\begin{rem}\label{R:Tmin}
Before giving the proof, we recall a consequence  of M\"oller's result. Suppose that $X$ has an automorphism $T$ with $T^* \omega =\xi^{-1} \omega$, where $\xi$ is a primitive $k$-th root of unity. Then 
$$\xi^{-1} \circ \pi \circ T : (X,\omega) \to (X_{\operatorname{min}}, \xi_{\operatorname{min}}^{-1}\omega_{\operatorname{min}})$$ is another translation covering from $(X,\omega)$ to a surface of the same genus, where in the definition of the covering $\xi^{-1}$ denotes  a rotation map $\Xmin \to  (X_{\operatorname{min}}, \xi_{\operatorname{min}}^{-1} \omega_{\operatorname{min}})$. The uniqueness of this map implies there is a map $t: X_{\operatorname{min}}\to X_{\operatorname{min}}$ satisfying $t^* \omega_{\operatorname{min}} = \xi^{-1} \omega_{\operatorname{min}}$ and $\pi \circ T = t \circ \pi$. This $t$ might have smaller order than that of $T$, but of course the order of $t$ is a multiple of $k$. 
\end{rem}

\begin{proof}
If $(X,\omega)$ does not admit any translation covering maps to strictly half translation surfaces, we may set $\Qmin=\Xmin$. 

So suppose there is a translation covering map $h: (X,\omega)\to (Q', q')$, where $(Q',q')$ is not the square of an Abelian differential. 
Recall  that any translation covering map  from a translation surface to a quadratic differential lifts to a map from the translation surface to the square root of the quadratic differential. Let $(X',\omega') \to (Q',q')$ be the square root of $(Q',q')$, and let $J$ be the involution on $(X',\omega')$ so $(Q',q')=(X',\omega')/J$. By the defining property of $\Xmin$, there exists a translation covering map $\pi: (X',\omega')\to \Xmin$,  and by the previous remark we get an automorphism $j$ negating $\omega_{\operatorname{min}}$ and satisfying $\pi = j \circ \pi \circ J$. 


Since $\Xmin$ does not cover a smaller genus translation surface,  we get that  $j$ must be an involution, and  also that  $\Xmin$ has at most one involution negating $\omega_{\operatorname{min}}$. Hence the lemma is true with $\Qmin=\Xmin/j$. 
\end{proof}

We define a point in a point marking to be free if it can be moved freely, independently of the unmarked surface and the position of the other points in the point marking.

 \begin{rem}\label{R:main}
 We now clarify the structure of irreducible point markings, justifying our previous remark that they arise from covering constructions.
 
Given $\cM$, let $\cM_{\operatorname{min}}$ denote the set of all $\Qmin$ arising from all $(X, \omega)\in \cM$. There is a natural point marking $\cM_{\operatorname{min}}^{\mathrm{br}}$ over $\cM_{\operatorname{min}}$ which consists of all $(Q_\mathrm{min}, q_{\mathrm{min}}, B)$ such that there is a cover $(X, \omega)\to (Q_\mathrm{min}, q_{\mathrm{min}})$ branched over $B$, with $(X,\omega)\in \cM$ and $B$ of maximal size. 
If $\cN$ is an irreducible $n$-point  marking over $\cM$ with $n>1$, then for each $(X, \omega, S)\in \cN$, Theorem \ref{T:main} gives that $S$ maps to a single point $p\in (Q_\mathrm{min}, q_{\mathrm{min}})$. We can then consider the point marking $\cN_\mathrm{min}$ which consists of all $(Q_\mathrm{min}, q_{\mathrm{min}}, B\cup\{p\})$ that arise in this way. By Theorem \ref{T:main}, the point marking $\cN_\mathrm{min}$ over $\cM_\mathrm{min}$ consists of a number of $\cN_\mathrm{min}$-periodic points together with a number of free marked points. 

Note that in general $S$ could be a proper subset  of the fiber of $p$; we have $q_{\mathrm{min}}(S)=\{p\}$ but not always $S= q_{\mathrm{min}}^{-1}(q_{\mathrm{min}}(S))$. We also have that $p$ must be free; it cannot be an $\cN_\mathrm{min}$-periodic point.
 \end{rem}

\begin{thm}\label{T:cor}
If $x_1$ and $x_2$ are finitely blocked on $(X,\omega)$, then either they are both $\cM$-periodic points or zeros, where $\cM$ is the orbit closure of $(X,\omega)$, or $\piQ (x_1)=\piQ (x_2)$. 
\end{thm}

To prove this theorem, which is the main result of this subsection, we first require two lemmas. Given two points $x_1$ and $x_2$ that are finitely blocked by a collection of points $B$, let $\M_{x_1,x_2,B}$ be the $\GL_2(\R)$ orbit closure of $(X, \omega; p, q; B)$ in $\cH^{*n+2}$ where $n$ is the size of $B$. We permit the points $x_1$ and $x_2$ to coincide and to be zeros, in which case we use the same notation, but take the orbit closure in $\cH^{*n+1}$ or $\cH^{*n}$. Finally, in order to refer to specific zeros, we will work on a finite cover of $\M$ where the zeros are labeled. We will suppress these details in the sequel.

\begin{lemma}\label{L:OB}
If $(X', \omega'; x_1', x_2'; B')$ belongs to $\M_{x_1,x_2,B}$ then $x_1'$ and $x_2'$ are  blocked by $B'$.
\end{lemma}

\begin{proof}
The locus of $(X', \omega'; x_1', x_2'; B')$ such that there is a straight line segment from $x_1'$ to $x_2'$ not intersecting $B'$ or the zeros of $\omega'$ is open and $\GL_2(\R)$ invariant. Therefore, its complement is closed and $\GL_2(\R)$-invariant and hence contains $\M_{x_1, x_2, B}$ since it contains $(X, \omega; x_1, x_2, B)$. 
\end{proof}

We define a blocking set to be minimal if no proper subset also blocks the two points. 

\begin{lemma}\label{L:MB}
Neither $x_1$ nor $x_2$ is free in $\M_{x_1,x_2,B}$. If $B$ is minimal, then locally in $\M_{x_1,x_2,B}$ the position of the points in $B$ are determined by the unmarked surface and $x_1,x_2$. 

Moreover, either 
\begin{enumerate}
\item  $\M_{x_1,x_2,B}$ consists entirely of periodic points, or 
\item  $\M_{x_1,x_2,B}$ consists of periodic points and one other irreducible piece which contains at least one of $x_1$ or $x_2$.
\end{enumerate}
\end{lemma}

\begin{proof}
If $x_1$ is free, we can move it into a small ball around $x_2$ that doesn't contain any points of $B$, and find a straight line segment from $x_1$ to $x_2$ not intersecting $B$.  

If some points in $B$ could be moved without changing the underlying unmarked surface or the position of $x_1,x_2$, we could move at least one of these points off the countable collection of line segments from $x_1$ to $x_2$ to obtain a smaller finite blocking set. This proves the first statement. 

Given the first statement, if the second statement did not hold, then $x_1$ and $x_2$ would have to be contained in different irreducible pieces that are not periodic points, and there could be no other irreducible pieces that are not periodic points. Pick a surface in $\M$, and move $x_1$ so it is not a periodic point. Then move $x_2$ to $x_1$. If $x_2$ collided with another point of its irreducible piece as it arrives at $x_1$, this would prove that $x_1$ is in fact a periodic point, by Lemma \ref{L:boundaryperiodic}. Hence $x_2$ can be moved very close to $x_1$ without another point of its irreducible piece moving close to $x_1$. This is a contradiction, because we can then find a straight line segment from $x_1$ to $x_2$ not intersecting $B$.  
\end{proof}

\begin{proof}[Proof of Theorem \ref{T:cor}]
Let $B$ be a minimal finite blocking set.  

First we suppose that $x_1$ is $\cM$-periodic or a zero and show that so is $x_2$. Indeed, move $x_2$ very close to $x_1$. Since all points in the point marking are either fixed or move with slope $\pm1$ with respect to $x_2$, there cannot always be a point of $B$ in between $x_2$ and $x_1$. This gives a contradiction to the assumption that $x_1$ and $x_2$ are blocked by $B$. 



Next suppose neither $x_1$ nor $x_2$ is periodic. By the previous lemma and  Theorem \ref{T:main} they must map to the same point under $\piQ $. 
\end{proof}

\begin{cor}\label{C:LMW}
A non-singular point on a translation surface that is not a torus-cover is only finitely blocked from finitely many other points.
\end{cor}
\begin{proof}
Let $p$ be a non-singular point on a translation surface that is not a torus cover. If $p$ is periodic then it is only finitely blocked from other periodic points, of which there are finitely many by  Theorem~\ref{T:periodic}. If $p$ is not periodic then it is only finitely blocked from other points in $\piQ^{-1} \left( \piQ(p) \right)$, of which there are only finitely many. 
\end{proof}

\subsection{$k$-differentials, $k>2$.}
Throughout this section we will suppose that $(S, \theta)$ is a Riemann surface $S$ with a meromorphic finite area $k$-differential $\theta$, $k>2$, and $\theta$ is not a power of a lower order differential. Let $(X, \omega)$ be the canonical unfolding of $(S, \theta)$ to an Abelian differential, which comes with a map $\pi_S:(X, \omega) \to (S,\theta)$. In this section we will prove the following:

\begin{thm}\label{T:FBK}
If $(X, \omega)$ is not a translation covering of a torus then there are only finitely many pairs of finitely blocked points on $(S, \theta)$.
\end{thm}

\begin{proof}
 $X$ has a rotational self-symmetry $T$ of order $k$ with $(S,\theta)=(X,\omega)/\langle T\rangle$. By Remark \ref{R:Tmin}, we see that it descends to an automorphism $t$ of $X_{\operatorname{min}}$, with $\piX \circ T= t\circ \piX$. 

If $s_1$ and $s_2$ are finitely blocked on $(S, \theta)$, then the set $\pi_S^{-1}(s_1)$ is finitely blocked from the set $\pi^{-1}(s_2)$, which means there is a finite set $B$ such that every straight line segment from one set to the other intersects $B$. Equivalently, every point of one set is finitely blocked from every point in the second set.

Suppose that $s_1$ or $s_2$ is such that $\pi_S^{-1}(s_1)$ and $\pi_S^{-1}(s_2)$ do not contain any $\cM$-periodic points and do not map to any of the finitely many points in $\Xmin$ that are fixed by a non-trivial power of $t$. (As usual, $\cM$ is the orbit closure of $(X,\omega)$.)

We will show that $s_1$ and $s_2$ are not finitely blocked. Suppose in order to find a contradiction that  they are. Consider a point of  $\pi_S^{-1}(s_1)$ and a point of $\pi_S^{-1}(s_2)$. Since these two points are finitely blocked, Theorem \ref{T:cor} gives that they map to the same point in $\Qmin$. Hence $\piQ $ maps $\pi_S^{-1}(s_1)$  to a single point of $\Qmin$. But $\pi_S^{-1}(s_1)$ is a $T$ orbit, so its image on $\Xmin$ must be a $t$ orbit of size at most two, which is a contradiction since $k>2$. 
\end{proof}

Recall that our convention is that polygons are assumed to have connected boundary.

\begin{prop}\label{P:toruscover}
A rational polygon unfolds to the cover of a torus if and only if the polygon is Gaussian, Eisenstein, or rectangle-tiled. 
\end{prop}

\begin{proof}
Let $P$ be a rational polygon and suppose its unfolding  $(X,\omega)$ is a torus cover; we will show $P$ is Gaussian, Eisenstein, or rectangle-tiled. (The other direction is easy.)  

Let $k$ be the least common denominator of the angles divided by $\pi$, so $X$ admits an order $k$ symmetry $T$ with $T^*(\omega)=\xi\omega$, where $\xi$ is a primitive $k$-th root of unity. 
The symmetry $T$ arises in the definition of the unfolding of $P$, and the quotient of $X$ by $\langle T\rangle$ can be understood  as two copies of the $P$ glued together to form a sphere, often called the pillowcase double of $P$. 

By assumption, $\omega$ lies in a two dimensional subspace of $H^1(X, \bC)$ defined over $\bQ$, spanned by $\omega$ and its complex conjugate. Since $T^*(\omega)=\xi\omega$, this subspace is invariant under $T^*$. Hence $T^*$ restricted to this rational subspace must have all Galois conjugates of $\xi$ as eigenvalues. Hence, the degree of $\bQ(\xi)$ as a field extension of $\bQ$ is at most 2, and we conclude that $k\in \{2,3,4, 6\}$.

For each zero of $\omega$ there is some non-trivial power of $T$ that fixes it.  Indeed, a zero of $\omega$ that was not fixed by any non-trivial power of $P$ would give rise to a cone angle of the pillowcase double of angle at least $4\pi$, which is a contradiction because we have assumed that $P$ does not have any slits.  Hence if $p:H^1(X,\Sigma, \bC)\to H^1(X,\bC)$ is the usual map from cohomology relative to the set $\Sigma$ of zeros of $\omega$ to absolute cohomology, we get that $T$ acting on $\ker(p)$ does not have any primitive $k$-th roots of unity as eigenvalues. (Here we count preimages of corners of $P$ of angle $\frac{\pi}{k}$ as zeros of order zero and include them in $\Sigma$.) Indeed, $\ker(p)$, viewed as a representation of $\bZ/n$ via the action of $T$, is a sub-representation of a direct sum of representations of $\bZ/n$ that factor through smaller groups.  Hence the dimension of the $\xi$-eigenspace of $T$ is the same in absolute and relative cohomology. 

The sum of primitive eigenspaces of $T$ in relative and absolute cohomology are both defined over $\Q$. Since $p$ induces a $\Q$-linear isomorphism between them it follows that the relative periods of $\omega$ are rational linear combinations of the absolute periods of $\omega$. 


If $k\in \{3,4, 6\}$ then the periods span a lattice in $\bC$, and after rotating and scaling the relative periods lie in $\bQ[\xi]$.  Since the sides of $P$ are all periods, we can, up to rescaling, assume that the vertices of $P$ lie in $\bZ[\xi]$. We can also assume that some edge of $P$ is horizontal, by rotating and scaling. Since all angles of $P$ are multiples of $\pi/k$, we get that all edges have angle a multiple of $\pi/k$ with horizontal. 

First suppose $k=4$. Then the vertices of $P$ are in $\bZ[i]$. Starting with the usual 1 by 1 square grid, we can subdivide each square into 8 triangles so that the center of the square and the center of each edge, as well as the four corners of the square, are the vertices. Then it is easy to see that all edges of $P$ are unions of edges of this triangulation of the plane, and hence that $P$ is Gaussian. 

Next suppose $k\in \{3,6\}$. Then the vertices of $P$ are vertices in the usual tiling of the plane by equilateral triangles. We can similarly subdivide each equilateral triangle into six triangles, and again get that all the edges of $P$ are unions of edges of this triangulation of the plane, and hence that $P$ is Eisenstein.

If $k = 2$ then after rotating and scaling the relative periods lie in $\Z[ia]$ for some real number $a$ and the polygon is rectangle-tiled. 
\end{proof}

\begin{proof}[Proof of Theorem \ref{T:poly}]
If $P$ is Gaussian or Eisenstein, then it unfolds to a torus cover, where it is known that any two points are finitely blocked. Hence any two sets of points are finitely blocked (just take the union of the blocking sets). 

Suppose now that $P$ is not Gaussian or Eisenstein. By Proposition~\ref{P:toruscover}, $P$ does not unfold to a torus cover and so there are only finitely many periodic points on the unfolding and any point is finitely blocked from only finitely many others by Theorem~\ref{T:cor}. If some angle is not an integer multiple of $\frac{\pi}{2}$ then the pillowcase double of $P$ is a $k$-differential for $k > 2$ and so Theorem~\ref{T:FBK} implies that there are only finitely many pairs of finitely blocked points on $P$.
%
\end{proof}

\subsection{Prime triangles.}
Theorem \ref{T:main} can sometimes be applied without knowing the orbit closure of a translation surface, since both flat and algebro-geometric methods exist to restrict the number of periodic points on a translation surface without knowing the orbit closure. Here is one example. 

\begin{thm}\label{T:prime}
Consider a triangle with angles $\frac{a}{\ell}\pi, \frac{b}{\ell}\pi, \frac{c}{\ell}\pi$ with $\ell>3$ prime and $\{a, b, c\} \ne \{1, 2, 4\}$. If two points are finitely blocked they are both vertices and the minimal blocking set is empty in the non-isosceles case and $\{m\}$ where $m$ is the midpoint of the line joining the two vertices of equal angle in the isosceles case.
\end{thm} 

\begin{rem}
The authors have verified that the result still holds for the $(1,2,4)$ triangle, but have chosen to omit the proof.
\end{rem}
\begin{rem}\label{R:permission}
We permit billiard paths to run along the edge of the table, but not to pass through vertices of the polygon. Using the description of the finite blocking set, we get moreover that two vertices with different angle are not finitely blocked. The statement allows for the possibility that a vertex is finitely blocked from itself. 
\end{rem}


 If $\ell=3$, the triangle unfolds to a torus and any two points are finitely blocked. 
 
 We require the following deep result due to Tzermias~\cite[Theorem 1.1]{Tzermias} in the nonhyperelliptic case and Grant-Shaulis~\cite[Theorem 1.1]{Grant-Shaulis} in the hyperelliptic case and which builds on work of Coleman~\cite{Coleman-etale} and Coleman-Tamagawa-Tzermias~\cite{CTT-Fermat}.

\begin{thm}\label{T:prime2}
Let $(X,\omega)$ be the unfolding of a triangle with angles $\frac{a}{\ell}\pi, \frac{b}{\ell}\pi, \frac{c}{\ell}\pi$ with $\ell>5$ prime and so that $(a, b, c) \ne (1,2,4)$. The only points of $(X,\omega)$ whose difference from a branch point is torsion are branch points and, when $(X,\omega)$ is hyperelliptic, Weierstrass points. 
\end{thm}

\begin{proof}[Proof of Theorem \ref{T:prime}.]
By work of Filip, the difference between any two periodic  points of $(X,\omega)$ must be torsion in the Jacobian \cite{Fi2}. (In general, Filip allows for more complicated twisted torsion relations, but to have non-trivial twisting one must consider relations between at least 3 points. In general, Filip also allows for the difference to be merely torsion in a factor of the Jacobian, but since $\ell$ is prime the relevant factor is in fact the whole Jacobian.)

When $\ell$ is prime,  $\omega$ does not lie in any proper rationally defined subspace of $H^1(X,\bC)$, because $\omega$ lies in an eigenspace whose Galois conjugates span $H^1(X,\bC)$. It follows that $(X,\omega)$ does not cover a translation surface of smaller genus. So $(X,\omega)=\Xmin$. Furthermore, if $X$ has an involution negating $\omega$, then this involution must be hyperelliptic, since its minus one eigenspace is rationally defined. Hence  $X$ is hyperelliptic or $X=Q_{\operatorname{min}}$.  By Theorem~\ref{T:cor} the only pairs of finitely blocked points are points that unfold to periodic points.  

Suppose first that $(X, \omega)$ is not hyperelliptic. When $\ell > 5$  and $\{a, b, c\} \ne \{1, 2, 4\}$, Theorem~\ref{T:prime2} implies that the only points that unfold to periodic points are vertices of the triangle.  By Lemma~\ref{L:MB}, a minimal blocking set of two finitely blocked periodic points consists of only periodic points. 

Now suppose $(X,\omega)$ is hyperelliptic. This happens if and only if the triangle is isosceles (see for example \cite[Section 4]{Coleman-etale}). When $\ell > 5$, Theorem~\ref{T:prime2} states the only points that unfold to periodic points (aside from the vertices) are points that unfold to Weierstrass points. The only such point on an isosceles triangle is the midpoint $m$ of the edge between the two vertices of equal angle.  Indeed, since the hyperelliptic involution is central and is a flat isometry, it descends to an involution of the pillowcase double. (Since $\ell$ is odd, the hyperelliptic involution cannot be part of the deck group.) This involution can be viewed as an automorphism of $\bP^1$ exchanging two points corresponding to the two equal angles and fixing a third point, and there is only one such involution.  

By Lemma~\ref{L:MB}, a minimal blocking set of two finitely blocked periodic points consists of only periodic points. Therefore, $m$ is not finitely blocked from any of the vertices of the triangle. The trajectories shown in Figure~\ref{F:Fagnano} shows that $m$ is not blocked from itself. Therefore, in the hyperelliptic case we have also shown that the only finitely blocked points are pairs of vertices. 
\begin{figure}[h!]
    \begin{subfigure}[b]{0.4\textwidth}
        \centering
        \resizebox{.8\linewidth}{!}{\begin{tikzpicture}        		
        		\draw (0,0) -- (2, 2) -- (4, 0) -- (0,0);
              	  \draw[dashed] (2,0) -- (1,1) -- (3,1) -- (2,0);
	\node at (.4, .2) {$\theta$}; \node at (3.6, .2) {$\theta$}; \node at (1.4, 1.2) {$\theta$}; \node at (2.6, 1.2) {$\theta$};
		\draw[black, fill] (2, 0) circle[radius = 1.6pt]; \node at (2, -.2) {$m$};
           \end{tikzpicture} }
            \label{SF:blah1}
            \caption{ $\theta > \pi/4$ }
    \end{subfigure}
        \qquad
        \begin{subfigure}[b]{0.4\textwidth}
        \centering
        \resizebox{.8\linewidth}{!}{\begin{tikzpicture}
        \draw (0,0) -- (2, 2) -- (4, 0) -- (0,0);
           \draw[dashed] (1,1) -- (2,0) -- (3,1);
	\node at (.4, .2) {$\theta$}; \node at (3.6, .2) {$\theta$}; \node at (1.6, .2) {$\theta$}; \node at (2.4, .2) {$\theta$};
		\draw[black, fill] (2, 0) circle[radius = 1.6pt]; \node at (2, -.2) {$m$};
            \end{tikzpicture} }
              \label{SF:blah2}
             \caption{ $\theta = \pi/4$ }
    \end{subfigure}
    \qquad
        \begin{subfigure}[b]{0.4\textwidth}
        \centering
        \resizebox{.8\linewidth}{!}{\begin{tikzpicture}
        \draw (0,0) -- (2, 2) -- (4, 0) -- (0,0);
           \draw[dashed] (1,0) -- (1,1) -- (2,0) -- (3,1) -- (3,0);
	\node at (.4, .2) {$\theta$}; \node at (3.6, .2) {$\theta$}; \node at (1.2, .4) {$2\theta$}; \node at (2.8, .4) {$2\theta$};
		\draw[black, fill] (2, 0) circle[radius = 1.6pt]; \node at (2, -.2) {$m$};
            \end{tikzpicture} }
              \label{SF:blah22}
             \caption{ $\theta < \pi/4$ }
    \end{subfigure}
\caption{The degenerate Fagnano trajectory in an isosceles triangle}
\label{F:Fagnano}
\end{figure}
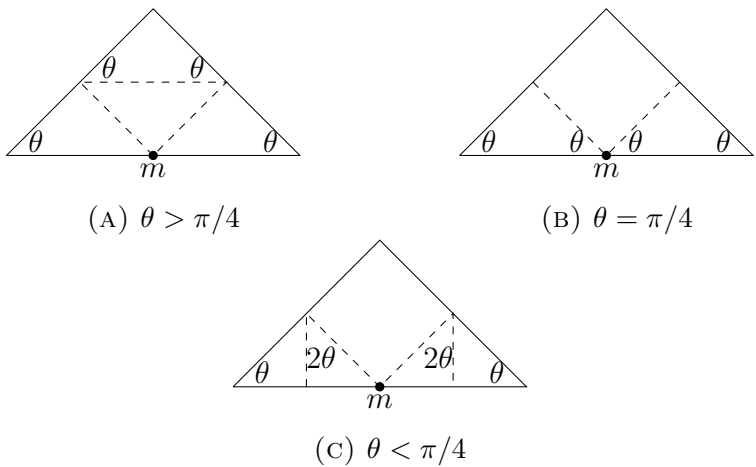

Both triangles with $\ell = 5$ unfold to Teichm\"uller curves in genus two and hence M\"oller~\cite{M2} implies that the only periodic points are zeros and Weierstrass points. Hence the above arguments hold also in this case. 
\end{proof}

\subsection{Description of blocking sets.}
Assume that $(X,\omega)$ is not a torus cover. 

\begin{thm}\label{T:blockingsets}
Let $x_1$ and $x_2$ be finitely blocked on $(X,\omega)$, and let $B$ be any minimal blocking set. If $x_1$ and $x_2$ are periodic points, then so are all points in $B$. Otherwise, one of the following holds:
\begin{enumerate} 
\item  $\piX(x_1)=\piX(x_2)$,  $B$ does not contain any periodic points and $\piX$ maps $\{x_1,x_2\}\cup B$ to a single point. 
\item  $\piX(x_1)\neq \piX(x_2)$ but $\piQ (x_1)=\piQ (x_2)$, and if  $B'$ is the set of non-periodic points in $B$, then $\piQ $ maps $\{x_1,x_2\}\cup B'$ to a single point.  
\end{enumerate}
\end{thm}

\begin{rem}
In the case of two periodic points and the case of two points that are identified under $\piX$, the converse - i.e. that any such points $x_1,x_2$ must be finitely blocked - is false. However, in the third case, the converse follows  from Lemma~\ref{L:involution-blocking}, which shows that $x_1$ and $x_2$ are finitely blocked by the preimages under $\piX$ of fixed points of the involution on $X_{min}$. In particular, a minimal blocking set is contained in the set of periodic points in this case.
\end{rem}

\begin{proof}[Proof of Theorem \ref{T:blockingsets}]
If both $x_i$ are periodic, the result follows from Lemma \ref{L:MB}. So assume at least one of $x_1$ or $x_2$ is not periodic. 

 Theorem \ref{T:cor} gives that $\piQ (x_1)=\piQ (x_2)$, and Lemma \ref{L:MB} and Theorem \ref{T:main} give that the blocking sets consists of periodic points and the $\piQ $ fiber of $x_1$. This proves the result if $\piX(x_1)\neq \piX(x_2)$. 

So assume $\piX(x_1)=\piX(x_2)$. Any line segment from $x_1$ to $x_2$ maps under $\piX$ to a periodic line on $\Xmin$. Moving $x_1$ slightly, we can assume that $\piX(x_1)$ is not on the central core curve of any cylinder. Hence any image of $\piX(x_1)$ under the involution is not on one of these periodic lines through $\piX(x_1)$, so we may assume that $B'$ maps to $\piX(x_1)$. (By Theorem \ref{T:main}, every point of $B'$ maps to either $\piX(x_1)$ or its image under the involution (if there is an involution)). 

Similarly, moving $\piX(x_1)$ slightly we can assume it does not lie on any of the countably many periodic lines through periodic points on $\Xmin$, and so we get that $B$ contains no periodic points. 
\end{proof}

For  developments on the illumination and finite blocking problems subsequent to the completion of the present paper, see \cite{ApisaWrightGemini,Wolecki}.

%
%

\section{Background}\label{S:background}

Here we recall some background that will be used in the rest of the paper. 

\subsection{Affine invariant submanifolds.}
Given an affine invariant submanifold $\M$ and a point $(X, \omega)$ in $\M$ the tangent space\footnote{Formally, an affine invariant submanifold is a properly immersed submanifold in the stratum, and the image of this immersion may have self-crossings. At such a self-crossing, the tangent space depends on not just the surface $(X, \omega)$ in the stratum but also a point in the abstract manifold $\cM$. See \cite{LNW} for more details. For notational simplicity we will use notation adapted to the case when the image of $\cM$ has no self-crossings and hence $\cM$ can be identified with its image in the stratum.}  $T_{(X, \omega)} \M$ is naturally identified with a subspace of $H^1(X, \Sigma; \bC)$ where $\Sigma$ is the zero set of $\omega$. Let $p: H^1(X, \Sigma; \bC) \to H^1(X; \bC)$ be the natural map from relative to absolute cohomology. The rank of $\cM$ is defined as $\rank(\cM)=\frac12 \dim_\bC p(T_{(X,\omega)} \cM)$ for any $(X,\omega)\in \cM$. This is an integer by work of Avila-Eskin-M\"oller \cite{AEM}. 

The affine field of definition $\bk(\cM)$ of $\cM$ is the smallest subfield of $\bR$ such that $\cM$ can locally be defined by linear equations in period coordinates with coefficients in this field \cite{Wfield}. It is an algebraic extension of $\bQ$ of degree at most $\deg(\bk(\cM))\leq g$, where $g$ is the genus. 

We will use the matrices 
$$u_t=\left(\begin{array}{cc} 1&t\\0&1\end{array}\right), \quad \quad a_t=\left(\begin{array}{cc}1&0\\0& e^t\end{array}\right), \quad \quad
r_t=\left(\begin{array}{cc} \cos(t)&-\sin(t)\\\sin(t)&\cos(t)\end{array}\right).$$

We will refer to a cylinder on a translation surface $(X,\omega)$ together with a choice of orientation of its core curve as an oriented cylinder. Given a collection of parallel oriented cylinders, we will say they are consistently oriented if the holonomies of $\omega$ along the oriented core curves are positive multiples of each other. 

Given an oriented cylinder $C$ on a translation surface $(X,\omega)$,  we define $u_t^C(X,\omega)$ and $a_t^C(X,\omega)$ to be the result of the following process. Rotate $(X,\omega)$ so that $C$ becomes horizontal and the orientation is in the positive real direction, apply $u_t$ or $a_t$ respectively to just $C$ and not to the rest of the surface, and then apply the inverse rotation. Given a collection $\cC=\{C_1, \ldots, C_k\}$ of parallel consistently oriented cylinders, define $u_t^{\cC}(X,\omega)=u_t^{C_1} \circ \cdots \circ u_t^{C_k}(X,\omega)$ and $a_t^\cC(X,\omega)=a_t^{C_1} \circ \cdots \circ a_t^{C_k}(X,\omega)$. We refer to $u_t^{\cC}$ as the cylinder shear and $a_t^\cC$ as the cylinder stretch. Typically, either a choice of orientation for the cylinders will be clear, or else either choice will be equally good. 

Let $\cM$ be an affine invariant submanifold. We say that two cylinders $C_1, C_2$ on a surface $(X, \omega)\in \cM$ are $\cM$-parallel if they are parallel and remain parallel on nearby\footnote{If $\cM$ has self-crossings, then one considers only deformations arising from a neighbourhood in the abstract manifold $\cM$.} surfaces in $\cM$. These definitions were introduced in \cite{Wcyl}, where the following is shown. 

\begin{thm}[Cylinder Deformation Theorem]\label{T:CDT}
If $\cC$ is an equivalence class of $\cM$-parallel cylinders on $(X,\omega)\in \cM$, then $u_t^{\cC}\circ a_s^\cC(X,\omega)\in \cM$ for all $s,t\in \bR$. 
\end{thm}

If $\cC$ is as above and contains a saddle connection perpendicular to the core curves, we define the ``collapse" of $\cC$ to be the limit of $a_s^\cC(X,\omega)$ as $s\to -\infty$. The condition that $\cC$ contains a perpendicular saddle connection connection is equivalent to the surface degenerating as $t\to -\infty$, and here we take the limit in the partial compactification described in \cite{MirWri}. If there is a unique $t_0$ (up to Dehn twists) so that $u_{t_0}^{\cC}(X,\omega)$ contains a saddle connection in $\cC$ perpendicular to the core curves, for example if $\cC$ is a single simple cylinder, then we define the cylinder collapse to be the limit of $a_s^\cC u_{t_0}^{\cC}(X,\omega)$ as $s\to -\infty$. (Recall a simple cylinder is one for which each boundary is a single saddle connection.) 

Cylinder deformations apply equally well to translations surfaces $(X, \omega, S)$ with marked points $S$. If we write $$(X', \omega', S')=\lim_{s\to -\infty} a_s^\cC(X,\omega, S),$$ then the set $S'$ may have a different size than $S$. In particular, $S'$ maybe be non-empty even when $S$ is empty \cite{MirWri}. In general, $(X', \omega', S')$ might also have multiple components, however in all instances in this paper it will have only a single component.

\subsection{Finiteness of periodic points.}\label{SS:EFW}
We now explain  why Theorem \ref{T:periodic} is a special case of results in \cite{EFW}. All theorems in \cite{EFW} apply  to affine invariant submanifolds in $\cH^{*n}$, as well as those in strata without marked points. If $\cN$ is a $\cM$-periodic point, it is in particular an affine invariant submanifold of $\cM^{*1}$ (the preimage of $\cM$ in $\cH^{*1}$). All of $\cM, \cM^{*1},$ and $\cN$ have the same rank. By assumption, we have that $\cM^{*1}$ is either higher rank (i.e. rank greater than 1) or that the degree of affine field of definition is greater than 1 (or both), since otherwise $\cM$ would consist of torus covers. By \cite[Theorem 1.5]{EFW}, such an affine invariant submanifold cannot properly contain infinitely many affine invariant submanifolds of the same rank.

\subsection{Strata of quadratic differentials.}

\begin{lemma}\label{L:q-rank}
Let $\cQ(\kappa)$ where $\kappa = (k_1, \hdots, k_n)$ be a stratum of quadratic differentials. Let $m_{odd}$ be the number of odd numbers in $\kappa$ and $m_{even}$ the number of even numbers. Let $g$ be the genus of the Riemann surfaces on which the quadratic differentials lie. The rank and rel of the component is then
\[ \mathrm{rk}(\cQ) = g + \frac{m_{odd}}{2} - 1 \qquad \text{and} \qquad \mathrm{rel}(\cQ) = m_{even}. \]
\end{lemma}

We define the rel to be the dimension minus twice the rank.

\begin{proof}[Proof sketch.]
The rank is the difference of the genera of surfaces in $\cQ$ and their double covers, which can be computed using Riemann-Hurwitz formula. See \cite[Section 2.1]{KZ} for the formula for $\dim \cQ$. 
\end{proof}

\begin{cor}\label{C:q-rank-one}
The only rank one strata of strictly quadratic differentials are $\cQ(-1^4)$, $\cQ(2,-1^2)$, and $\cQ(2^2)$. 
\end{cor}




\subsection{Hat homologous saddle connections.}

Two saddle connections or cylinders on $(Q, q)\in \cQ$ are called hat homologous if they are parallel and remain so on all nearby surfaces in $\cQ$. 
 Configurations of hat homologous saddle connections were classified in \cite{MZ}; in particular,  two hat homologous cylinders must have  ratio of lengths in $\{\frac12, 1, 2\}$. 

We say that a quadratic differential is generic in a given direction if any two saddle connections in that direction are hat homologous. A cylinder in a half-translation surface is an isometric map of $\bR/(c\bZ) \times (0,h)$ into the surface. This always extends to a continuous map of $\bR/(c\bZ) \times [0,h]$ into the surface. The two boundary components of the cylinder are the images of $\bR/(c\bZ) \times \{0\}$ and $\bR/(c\bZ) \times \{h\}$. The multiplicity of a saddle connection on the component of the boundary  corresponding to $\bR/(c\bZ) \times \{0\}$ is the number of preimages of a point in this saddle connection in $\bR/(c\bZ) \times \{0\}$, and similarly for $\bR/(c\bZ) \times \{h\}$. This multiplicity is always 1 or 2. Define a simple cylinder to be one that has one saddle connection, with multiplicity one, in each of its boundaries.

We recall the following consequences of \cite[Theorems 1 and 2]{MZ}. 

\begin{prop}\label{P:MZ}
Suppose that $C$ is a cylinder in a generic direction on a quadratic differential. Then each boundary component of $C$ consists of either 
\begin{enumerate}
\item one saddle connection with multiplicity one, 
\item one saddle connection with multiplicity two, or
\item  two saddle connections, each with multiplicity one. 
\end{enumerate}
In the last case, removing the two saddle connections disconnects the surface, and the component not containing $C$ has trivial linear holonomy.  

\begin{figure}[h!]
\includegraphics[width=.3\linewidth]{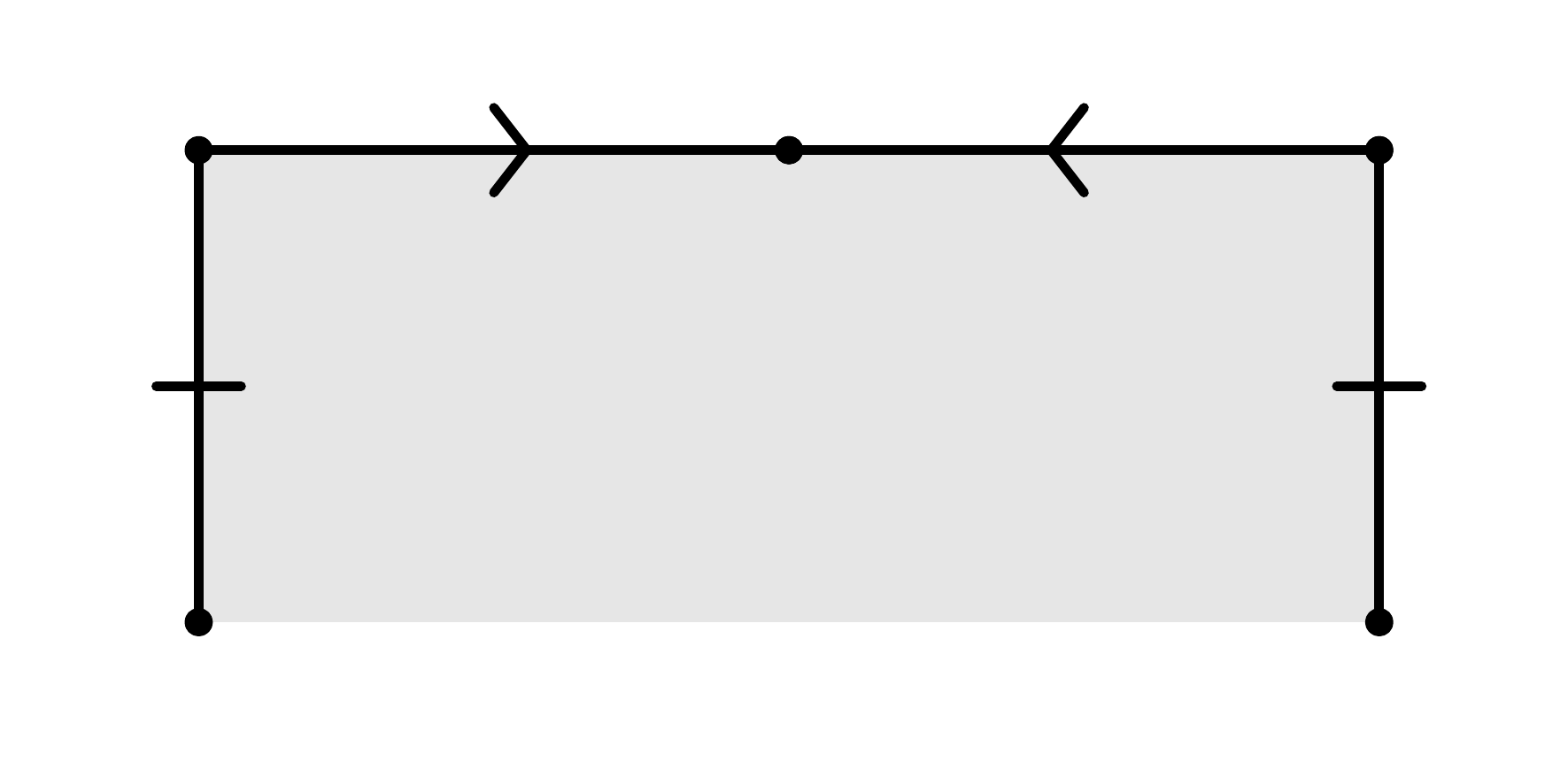}
\caption{The case of a multiplicity two saddle connection in Proposition \ref{P:MZ}.}
\label{F:MZ}
\end{figure}

Furthermore if $C$ shares a boundary saddle connection with another cylinder $C'$, then possibly after switching $C$ and $C'$ we have that $C'$ is simple and does not share a boundary saddle connection with any other cylinder, and $C$ has two saddle connections in the given boundary component as in case (3) above.  
\begin{figure}[h!]
\includegraphics[width=.9\linewidth]{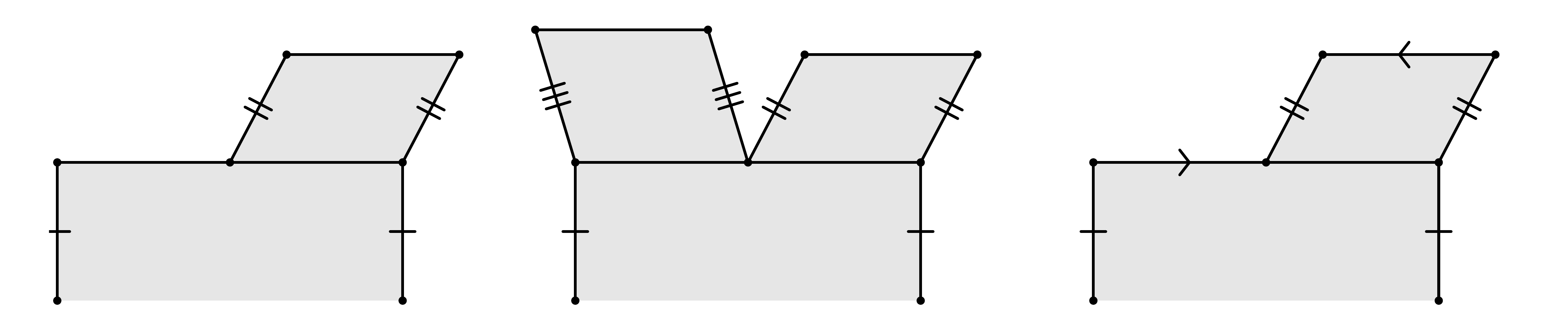}
\caption{The left and right images indicate the two possible configurations of $C$ and $C'$ in Proposition \ref{P:MZ}. The middle images reminds us that there may also be another cylinder adjacent to $C$. }
\label{F:3hat}
\end{figure}
 \end{prop}

\subsection{Primitivity}

A half-translation covering map from a quadratic differential $(Q, q)$ to a quadratic differential $(Q',q')$ is defined to be a branched covering $f: Q \to Q'$ of Riemann surfaces such that $q = f^*(q')$. So the definition is essentially identical to the definition of translation covering map, but the word ``half" reminds us that the domain might not be an Abelian differential.

\begin{lemma}\label{L:R-H}
The generic element of a component of a stratum of Abelian or quadratic differentials admits a non-bijective half-translation cover to another translation or half-translation surface if and only if the component is hyperelliptic. If the component is hyperelliptic and has rank bigger than 1, then the quotient by the hyperelliptic involution is the only such map.
\end{lemma} 

This is related to, but stronger than, \cite[Theorem 1.1]{MAff}. Since we are not aware of a proof of the lemma in the literature, we sketch one way that the lemma can be verified. 

\begin{proof}
All rank 1 strata are hyperelliptic, so we assume the stratum has rank at least 2. In particular, the generic surface in the stratum doesn't cover a surface in $\cQ(-1^4)$. 

Suppose the generic element $(Q, q)$ of a stratum $\cQ$ is a cover of $(Q_0,q_0)\in \cQ_0 \neq \cQ(-1^4)$ of degree $d>1$. It follows that every element of $\cQ$ is a cover of a surface in $\cQ_0$. 

We define an envelope to be a cylinder that has one boundary that consists of one  saddle connection of multiplicity two; see Figure \ref{F:MZ}. An envelope is simple if it also has one boundary that only contains a multiplicity one saddle connection.

We first claim $d=2$. It is possible to find a surface $(Q_0, q_0)\in \cQ_0$ with a  cylinder $C_0$ that is a simple cylinder or a simple envelope; for example, this follows from Lemma \ref{L:firstcyl}. We can assume that $C_0$ doesn't contain any branch points by simply moving them out of $C_0$. 

In particular $C_0$ has a zero (that isn't a pole) on at least one side. That side must consists of a single multiplicity one saddle connection, and its preimage consists of $d$ saddle connections of the same length. 

Suppose first that $C_0$ is a simple cylinder. In that case, there can't be more than one cylinder in the preimage, because cylinders in the preimage must have equal height and generic surfaces in a stratum don't have cylinders of the same height. So we get that the preimage is a single cylinder. This single cylinder has $d$ saddle connections on each side, giving $d=2$ by Proposition \ref{P:MZ}.

The case where $C_0$ isn't simple is a bit more subtle because a component of the preimage of $C_0$ can have twice the height of $C_0$. Using the arguments from before we reduce to the case where the preimage has one component, and is twice the height and twice the circumference, with two saddle connections on each side. Let $C'$ be any cylinder in the complement of the preimage. By Proposition \ref{P:MZ}, it can be seen to be simple. Repeating the argument using the image of $C'$ in the place of $C_0$  completes the proof that $d=2$. 

Given $d=2$, the lemma follows from a  Riemann-Hurwitz argument, using that if the generic element of one stratum covers an element of another, the dimension of the first stratum must be at least as large as the second, as in the determination of which strata have a hyperelliptic component \cite{KZ, LanneauHyp}. 
\end{proof}

\begin{rem}\label{R:Unique}
Given Corollary \ref{C:q-rank-one}, Lemma \ref{L:R-H} implies that the only genus $0$ stratum where every surface has an involution is $\cQ(-1^4)$. Lemma \ref{L:R-H} (plus a direct verification in $\cQ(2,2)$ and $\cQ(2,-1^2)$) shows that in every hyperelliptic stratum other than $\cQ(-1^4)$ and $\cH(\emptyset)$, the hyperelliptic involution is unique on a generic surface. (In genus zero or genus one a surface can have several hyperelliptic involutions.)
\end{rem}

%
%

\section{Proof of Theorem~\ref{T:Q}}\label{S:Q-overview}

Throughout the rest of this paper, $\cQ, \cQ'$, etc., will denote connected components of strata of quadratic differentials. Recall that point markings over strata of Abelian differentials are classified in \cite{Apisa}, so we make the following standing assumption for the remainder of the paper.  
\begin{ass}\label{A}
All strata of quadratic differentials considered will not consist  of squares of Abelian differentials.
\end{ass}

We begin by noting that we have already classified irreducible $n$-point markings with $n>1$, as a consequence of Theorem \ref{T:main} and Lemma \ref{L:R-H}. 

\begin{cor}\label{C:nothypbig}
Let $\cQ$ have rank bigger than 1. 
If $\cQ$ is not hyperelliptic, then there are no irreducible $n$-point markings with $n>1$. If $\cQ$ is hyperelliptic, the only such point markings occur when $n=2$ and the two points are interchanged by the involution. 
\end{cor}

Recall that, by Corollary \ref{C:q-rank-one}, the first requirement simply says that $\cQ$ is not $\cQ(-1^4)$, $\cQ(2,-1^2)$, or $\cQ(2^2)$.

\begin{proof}
One can rephrase Lemma \ref{L:R-H} as saying that if $(Q,q)$ is a generic element of a non-hyperelliptic $\cQ$, then $\Qmin=(Q,q)$, and otherwise $\Qmin$ is the quotient by the hyperelliptic involution. 
\end{proof}

The following proposition will provide the inductive step for our arguments. 
Recall that an envelope is a cylinder that has one boundary that consists of one  saddle connection of multiplicity two. An envelope is simple if it also has one boundary that only contains a multiplicity one saddle connection. 

In regards to Assumption \ref{A}, we remark that degenerating a simple cylinder or a simple envelope on a quadratic differential with non-trivial holonomy will not create a quadratic differential with trivial holonomy. (In contrast, degenerating a non-simple envelope may create a  quadratic differential with trivial linear holonomy.)

\begin{prop}\label{P:induct} 
Suppose that $(Q, q)\in \cQ$  has a simple cylinder $C$, and that degenerating $C$ gives $(Q', q', S)$, where  $(Q', q')\in \cQ'$ and $\cQ'$ does not have any periodic points. Then $\cQ$ does not have any periodic points. 

The same conclusion holds if $(Q,q)$ has a disjoint pair of simple envelopes, and degenerating either one similarly gives a $\cQ'$ without periodic points. 
\end{prop}

We will defer this proof of Proposition \ref{P:induct} to the end of the section.  A periodic point for a stratum of quadratic differential is a point marking of the same dimension, as in the case of Abelian differential. Each periodic point for $\cQ$ can be lifted to one for $\tilde{\cQ}$. (The fibers of the periodic point double in size, and we do not comment on whether the new periodic point is irreducible.) 

In the next two sections we will prove the following two results, which we will use in our proof of Theorem~\ref{T:Q}. Note that the three strata that appear in the next theorem are exactly the higher rank strata of minimal dimension (they have dimension 4 and rank 2), and that all rank 1 strata are hyperelliptic. 

\begin{thm}\label{T:findcyl} 
If $\cQ$ is non-hyperelliptic, and $$\cQ\notin \{ \mathcal{Q}(3, -1^3), \mathcal{Q}(5, -1), \mathcal{Q}(1, -1^5)\},$$ then there exists $(Q, q)\in \cQ$ with a simple cylinder, or a pair of disjoint simple envelopes, such that degenerating any one of these cylinders gives $(Q', q', S)$, where $(Q', q')\in \cQ'$ and $\cQ'$ is non-hyperelliptic.
\end{thm}

\begin{thm}\label{T:basecase} 
$\mathcal{Q}(3, -1^3),\mathcal{Q}(5, -1)$ and $\mathcal{Q}(1, -1^5)$ do not have periodic points. 
\end{thm}

\begin{proof}[Proof of Theorem~\ref{T:Q}]  
By using induction on $\dim \cQ$, Proposition \ref{P:induct} and Theorems \ref{T:findcyl} and \ref{T:basecase} immediately give the result when $\cQ$ is not hyperelliptic. 

When $\cQ$ is hyperelliptic, every surface in $\cQ$ covers a surface in a genus 0 stratum $\cQ_0$, which is not hyperelliptic and has the same rank. Let $n$ be the number of points that are not zeros or poles over which these covering maps are branched. (One can show $n\in \{0,1, 2\}$.) Any $\cQ$-periodic point gives rise to a periodic point over $\cQ_0^{*n}$.  This can be considered as a point marking over $\cQ_0$, and in this point marking the point arising from the $\cQ$-periodic point is not free. By Corollary \ref{C:nothypbig} this means we get a $\cQ_0$-periodic point, which is a contradiction. 
\end{proof}

\begin{rem}\label{R:QvsAb}
Theorem \ref{T:findcyl} is false without Assumption \ref{A}. Namely, if $\cQ=\cH^{odd}(4)$, then $\cQ$ is not hyperelliptic, but every degeneration of $\cQ$ is a genus two stratum of Abelian differentials and hence  hyperelliptic. Thus our proof does not apply as written to strata of Abelian differentials. 

However, a variant of the proof of Theorem \ref{T:findcyl}  shows that in a connected component of a stratum of Abelian differentials of genus at least 3 other than $\cH^{odd}(4)$ there is a surface with a simple cylinder that may be degenerated to produce a translation surface in a nonhyperelliptic connected component. A new proof of the classification of periodic points for Abelian differentials then follows almost verbatim to the proof for quadratic differentials. We omit the arguments for Abelian differentials.  
\end{rem}

Now we will proceed with the proof of Proposition \ref{P:induct}.  The difficulty of the proof is that if we naively collapse $C$, it may be that as $C$ decreases in size, the periodic point converges to a zero or pole, so that on the limit there is no periodic point. Notice that the assumptions imply that $\cQ$ is not rank one. 

\begin{proof}[Proof of Proposition \ref{P:induct}]
We will handle both cases simultaneously. In the first case, $C$ is the given simple cylinder, and in the second case, we let $C$ be either envelope. 
Suppose to a contradiction that $p$ is a $\cQ$-periodic point on $(Q, q)$. Let $\cC$ be the equivalence class of cylinders $\cQ$-parallel to $C$, and suppose without loss of generality that $\cC$ consists of horizontal cylinders and that the horizontal direction is generic. Our analysis will place increasingly strong constraints on $p$, until eventually we reach a contradiction. 

Suppose without loss of generality, after shearing the surface, that $C$ contains a vertical saddle connection. Recall we have defined the collapse of $C$ to be the limit of $a_s^C(Q,q)$ as $s\to -\infty$. Here it will be important that this limit can be thought of as limit of the path $a_{\log(t)}^C(Q,q)$, which is linear in period coordinates as $t$ ranges from $e$ to $0$. We refer to this whole path as the collapse path of $C$. 

The proof will proceed in three steps. First,  we will show that it suffices to show that we can move the marked point out of $\cC$. Second, we will show that if we can't collapse $C$ without $p$ merging with a singularity then the position of $p$ is entirely controlled by $C$, in a precise sense specified at the beginning of step 2. Finally, we will show that if the position of $p$ is entirely controlled by $C$ then $p$ can be moved out of $\cC$. 

\noindent \textbf{Step 1: $p$ must belong to a cylinder in $\cC$.} We claim that if $p$ is not contained in the interior or boundary of a cylinder of $\cC$, then its position in the complement of $\cC$ is unchanged by any cylinder deformation of $\cC$. To see this, first slightly  change the heights of $\cC$ so there is no rational relation between the moduli of the cylinders in $\cC$. By the Cylinder Deformation Theorem (Theorem \ref{T:CDT}), we know we can equally twist the cylinders in $\cC$ without changing the position of $p$. It follows from one of the easiest ingredients in the proof of Theorem \ref{T:CDT}, namely \cite[Corollary 3.4]{Wcyl}, that we can individually  twist each cylinder in $\cC$ without changing the position of $p$, because this twist is a limit of the deformations of $\cC$ guaranteed to exist by the Cylinder Deformation Theorem. See also \cite[Lemma 4.6]{MirWri}. The claim is proved.

Hence, if $p$ is disjoint from the closure of the cylinders in $\cC$, we may collapse $C$ and pass to a boundary surface $(Q', q', S\cup \{p\})$, where  $(Q', q')\in \cQ'$ and $\cQ'$ does not have any periodic points. Mirzakhani-Wright~\cite{MirWri} gives that $p\in (Q', q', S\cup \{p\})$ is not free, and by assumption $p$ cannot be a $\cQ'$ periodic point. But since $\cQ'$ does not have periodic points, $\cQ'$ cannot be hyperelliptic,  so Corollary \ref{C:nothypbig} implies that $p$ is either free or a $\cQ'$-periodic point, giving a contradiction.

Suppose now that $p$ lies on the boundary of $C$ for all time as $C$ collapses. Then we proceed as follows. If $C$ is a simple cylinder, then we proceed with the collapse and arrive at the same contradiction as before. If $C$ is an envelope, then it may not be possible to collapse $C$ without causing $p$ to coincide with a singularity on the boundary, see Figure~\ref{F:Shear1} (bottom). In this case, we relabel the cylinders so that the second envelope is labeled $C$.  (This is why we require two envelopes.) 

\begin{figure}[h!]
\includegraphics[width=.8\linewidth]{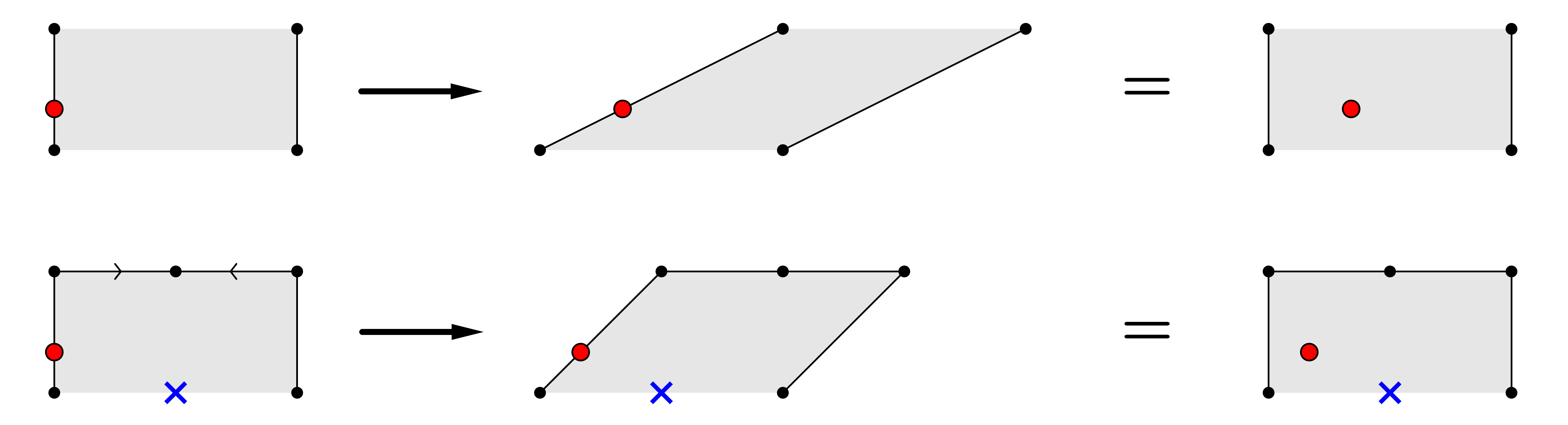}
\caption{Top: A full Dehn twist in a simple cylinder. Bottom: A half Dehn twist in an envelope. The bad position on the boundary is marked with an x. }
\label{F:Shear1}
\end{figure}

The proof is similar but easier if $p$ is on the boundary of another cylinder in $\cC$ for all time as $C$ collapses.

\noindent \textbf{Step 2: $C$ determines the position of $p$.}

We may now suppose that $p$ is contained in the interior of a cylinder $D\in \cC$, and either $D=C$ or the position of $p$ in the complement of $C$ is not constant along the collapse path. We will show that we may assume that there is a saddle connection joining $p$ to a zero on the boundary of $D$ whose holonomy is a fixed real multiple of a cross curve of $C$. We will also show $C\neq D$. Thus, this step could be more completely described as ``$D\neq C$, and $C$ determines the position of $p$ in $D$."

Shear $(Q, q)$ so that $p$ does not lie on a vertical separatrix (this could easily cause $C$ to no longer contain a vertical saddle connection). 
Recall the collapse path is defined to start at $t=e$ and end at $t=0$. Because we have assumed $p$ does not lie on a vertical separatrix, $p$ does not hit a singularity of the metric before $t=0$. We may partition the interval $(0,e)$ into closed subintervals according to which cylinder in $\cC$ (including the boundary of the cylinder) $p$ is in at a given time $t\in (0,e)$. (The subintervals overlap at their endpoints, and by convention we require adjacent intervals to correspond to distinct cylinders in $\cC$.) 

Let $\gamma$ be a path from a singularity to the periodic point. Let $f(t)$ denote the imaginary part of the period of $\gamma$ at time $t$ along the collapse path. (The function $f$ is only well defined up to replacing it with $-f$).   If $h$ is the height of the shortest cylinder in $\cC - \{ C \}$, then  $p$ passes through at most $\frac{|f(0) - f(e)|}{h}$ cylinders in $\cC - \{ C \}$ along the collapse path. This shows that the partition described in the previous paragraph is finite. 

By replacing $(Q,q)$ with an appropriate point on the collapse path, we may assume that in fact $p$  remains in $D$ along the collapse path. Giving up our assumption that $p$ doesn't lie on a vertical separatrix, with can replace $(Q,q)$ with a sheared version of itself to again assume without loss of generality that the cylinder $C$ contains a vertical saddle connection $v_1$. Note that the assumption that $p$ remains in $D$ is still valid.
 
Suppose first that $C=D$. If $p$ lies on a vertical separatrix contained in $C$ then we may shear the surface to perform a full Dehn twist (in the case that $C$ is simple) or half a Dehn twist (in the case that $C$ is an envelope) to ensure that $C$ still contains a vertical saddle connection and that $p$ does not lie on a vertical separatrix contained in $C$. See Figure~\ref{F:Shear1} and, for a non-example, Figure~\ref{F:Shear2}.  Collapsing $C$ now gives a half translation surface where $p$ is not a singularity of the metric, but is a periodic point, which is a contradiction. 

\begin{figure}[h!]
\includegraphics[width=.8\linewidth]{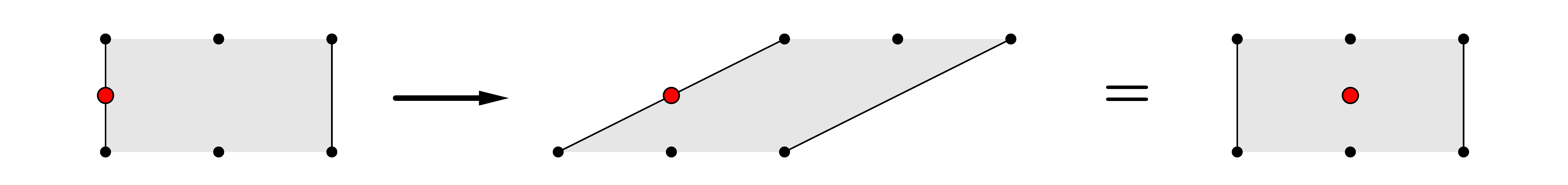}
\caption{Cylinders with four hat homologous boundary saddle connections must be avoided, since if $p$ is the midpoint of a vertical saddle connection, a Dehn twist cannot fix this problem. Another reason to avoid degenerating such cylinders is that the double cover may become disconnected (which can also happen for non-simple envelopes).}
\label{F:Shear2}
\end{figure}
 
Next we suppose that $C\neq D$.  Suppose that there is a small perturbation of $(Q, q)$, so that $v_1$ remains vertical and $C$ remains horizontal, and such that after replacing $(Q,q)$ with this deformation $p$ does not collide with a singularity at the conclusion of the collapse path. If this occurs, then, as above, $p$ becomes a periodic point on the boundary, which has no periodic points by hypothesis.  Therefore, we may suppose that after any small perturbation as above, $p$ collides with a singularity at the conclusion of the collapse path. In particular, this means that there is a saddle connection $v_2$ in $D$  joining a singularity on the boundary of $D$ to $p$, and a positive real constant $c$, so that the orbit closure of $(Q, q; p)$  locally satisfies  the equation $v_2 = c v_1$. 

\noindent \textbf{Step 3: $p$ may be moved out of $\cC$.}

The idea of the proof will now be to vary $v_1$ to move the periodic point $p$ outside of $\cC$ (which will be a contradiction). Since $v_2 = c v_1$ we see that by shearing $C$ to the left or right while fixing the complement of $C$ we may move the marked point to the left or right. Similarly, increasing the height of $C$ while fixing the rest of the surface causes $p$ to move up or down. We will use these two operations to force $p$ to move through the complement of $C$ (and we will be careful not to cause $p$ to enter $C$). We will refer in the sequel to these operations as ``moving $p$ using $C$".  Suppose that increasing the height of $C$ moves $p$ toward a boundary $B$ of $D$. 

Our approach will use the results of Masur-Zorich~\cite{MZ} that we summarized in Proposition \ref{P:MZ}. In the sequel, we will say that a saddle connection ``leads out of $\cC$" if it borders a cylinder in $\cC$ on exactly one side. 

\begin{sublem}\label{SL:MZ1}
The boundary $B$ cannot contain two saddle connections of multiplicity one. 
\end{sublem}

\begin{proof}
Suppose not to a contradiction. If both saddle connections in $B$ lead out of $\cC$ then we increase the height of $C$ to move $p$ through $B$ and out of $\cC$.  Otherwise, by Proposition \ref{P:MZ} the boundary $B$ of $D$ is as in one of the subfigures in Figure~\ref{F:3hatDB}

\begin{figure}[h!]
\includegraphics[width=.9\linewidth]{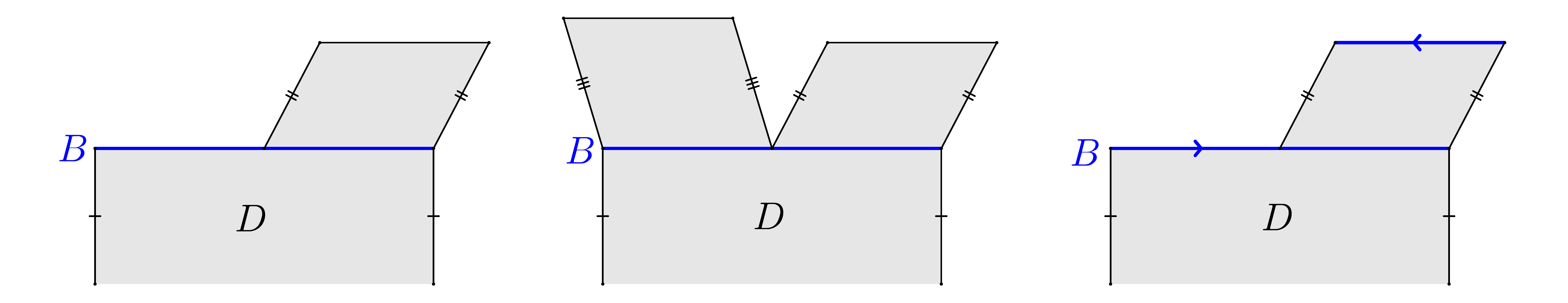}
\caption{Possible configurations of $D$ and $B$. In the left two cases, the upper edges are glued to the complement of $\cC$ (not shown).}
\label{F:3hatDB}
\end{figure}

We see that the left two configurations are impossible, as shown in Figure~\ref{F:LeavingEC1}. In that figure, we shear $C$ and increase its height so that the marked point leaves $\cC$. In the right configuration of Figure~\ref{F:LeavingEC1} this is easy since there is a saddle connection $b_0$ in $B$ that leads out of $\cC$. In the left configuration one of the two simple cylinders bordering $B$, call it $E$, is not $C$, and passing through $E$ leads out of $\cC$ by Proposition \ref{P:MZ}.

\begin{figure}[h!]
\includegraphics[width=.9\linewidth]{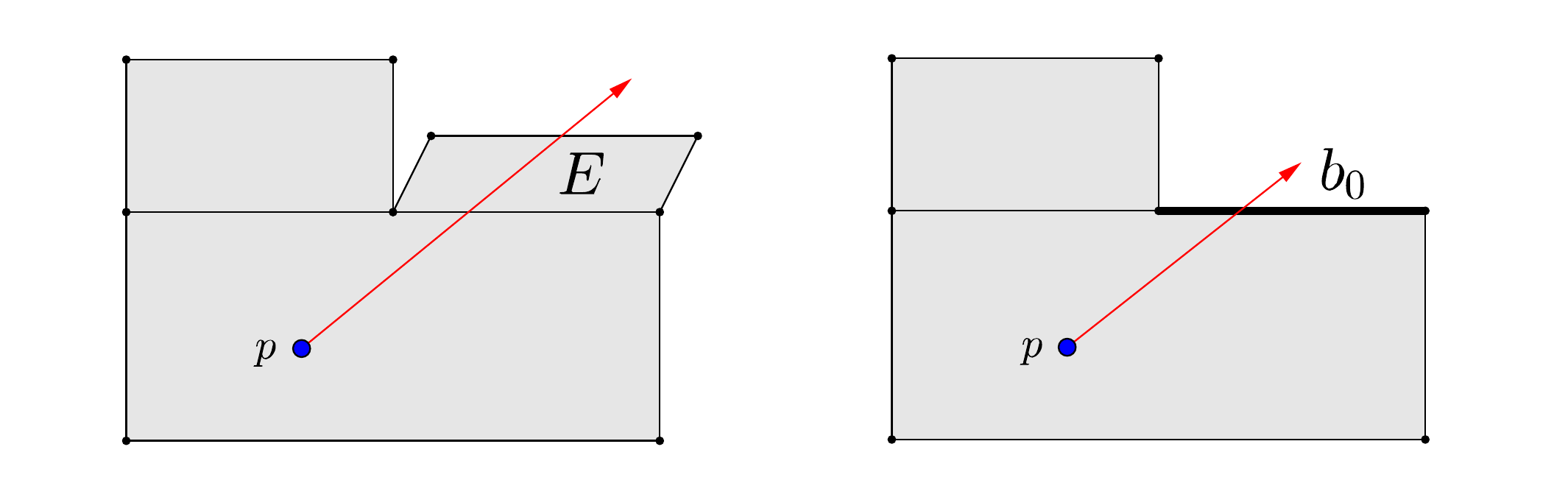}
\caption{Shear $C$ and increase its height so $p$ exits $\cC$}
\label{F:LeavingEC1}
\end{figure} 

Therefore, $D$ and its boundary $B$ are arranged as in Figure~\ref{F:D-prime-configuration}. Let $D'$ be the simple cylinder that borders $D$ along $B$.

\begin{figure}[h!]
\includegraphics[width=.4\linewidth]{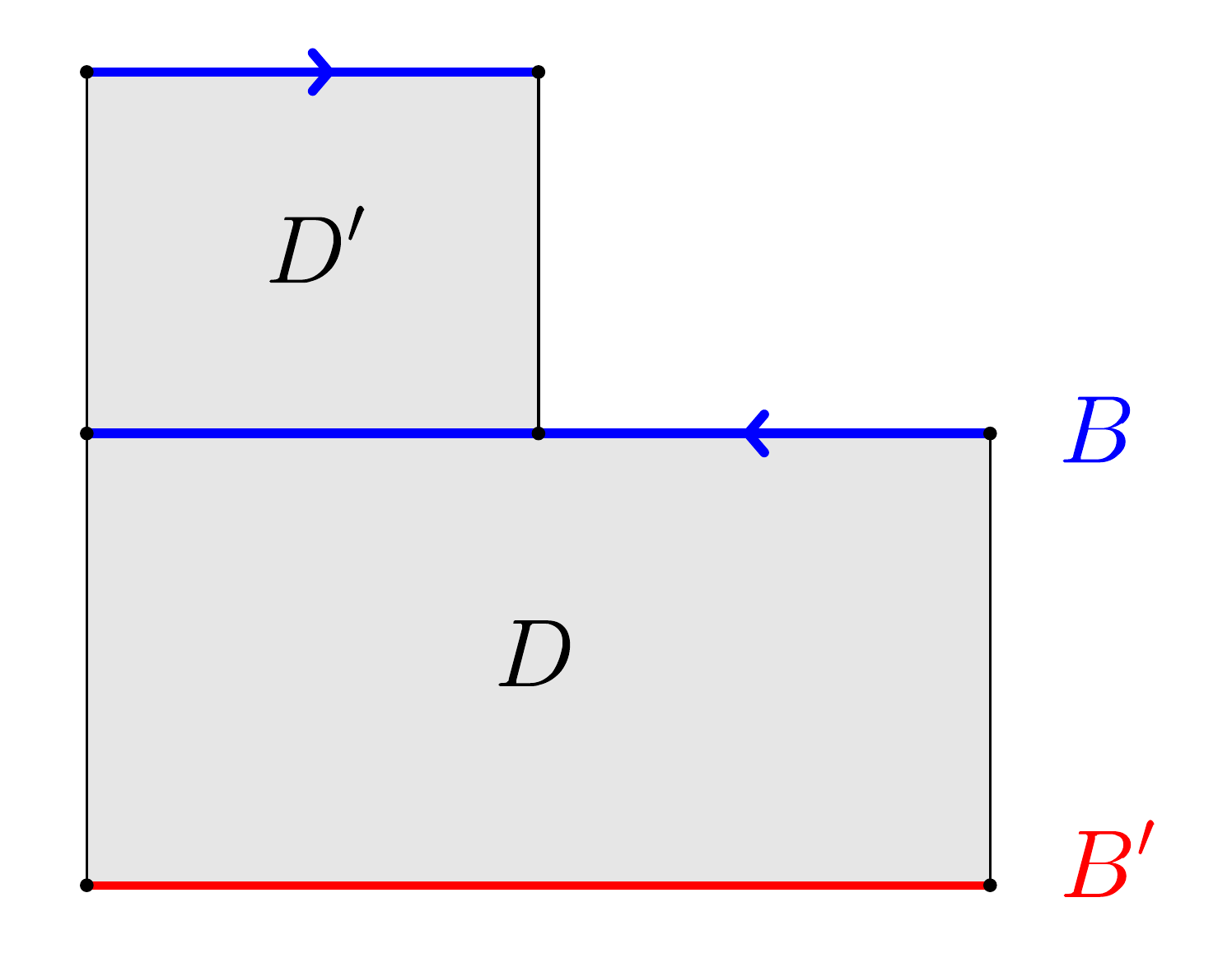}
\caption{The configuration of $D$ and its boundary $B$}
\label{F:D-prime-configuration}
\end{figure}

Let $B'$ be the boundary of $D$ opposite $B$. Notice that $B'$ cannot consist of a single multiplicity two saddle connection since then $\cQ = \cQ(2, -1^2)$, which is rank one. Moreover, $B'$ cannot consist of two multiplicity one saddle connections that bound a simple cylinder (as in Figure~\ref{F:D-prime-configuration}) since then $\cQ = \cQ(2, 2)$, which is also rank one. Therefore, Proposition \ref{P:MZ} implies that either $B'$ contains a saddle connection that leads out of $\cC$ or $B'$ borders two distinct simple cylinders in $\cC$. The possibilities are shown in Figure~\ref{F:LeavingEC2}. In that figure we see that when $D' \ne C$ it is possible to shear $C$ and increase its height to move the marked point $p$ out of $\cC$ as we did in Figure~\ref{F:LeavingEC1}. It is important that $D' \ne C$ since the marked point will pass through $D'$ as it moves to leave $\cC$.

\begin{figure}[h!]
\includegraphics[width=.9\linewidth]{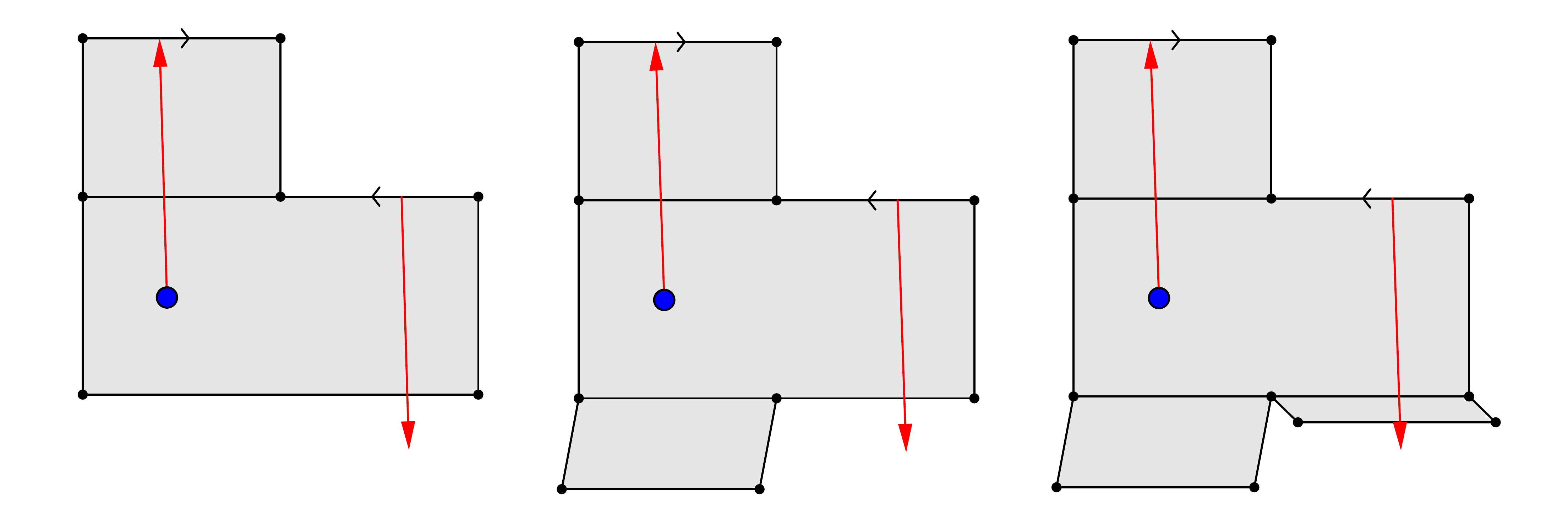}
\caption{Three configurations of $D$ and moving $p$ out of $\cC$ when $D' \ne C$. In the leftmost, the bottom saddle connection could also be two distinct saddle connections. }
\label{F:LeavingEC2}
\end{figure} 

Therefore, we may suppose that $D' = C$. Recall that there is a cross curve $v_1$ in $C$ and a positive real constant $c$ so that $v_2 = c v_1$ where $v_2$ is a saddle connection joining a singularity on $B'$ to $p$. We now decrease the height of $C$ so that the imaginary part of the period of $v_1$ changes from positive to negative. As we see in Figure~\ref{F:LeavingEC3}, this ``overcollapse of $C$" can be performed in a way that moves the marked point out of $\cC$ (in the rightmost figure we must first make the bottom cylinder that the marked point passes through sufficiently small; this may be achieved without affecting the position of the marked point since $v_2 = cv_1$). 

When $C$ overcollapses, a new horizontal cylinder $C'$ is created, visible at the top of the figures in Figure~\ref{F:LeavingEC3}. It is simple, with both boundary components glued to the cylinder $D\in \cC$ below it, so it is generically parallel to the surviving cylinders of $\cC$. (One might say more informally that $C'$ is ``in $\cC$", meaning that it is in the equivalence class of the surviving cylinders of $\cC$.) No other horizontal cylinders are created by the overcollapse, since it does not effect the rest of the surface.

\begin{figure}[h!]
\includegraphics[width=\linewidth]{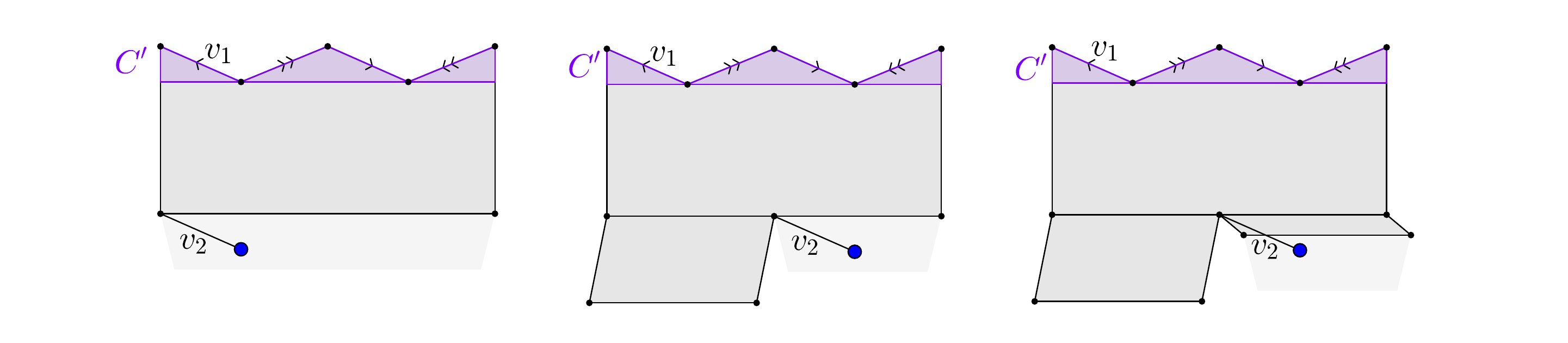}
\caption{Moving $p$ out of $\cC$ when $D' = C$. (This figure is drawn with $c=1$.)}
\label{F:LeavingEC3}
\end{figure}

Once $p$ has left $\cC$, we obtain a contradiction as follows. After $C$ is collapsed, the continuation of the overcollapse deformation of the unmarked surface is actually a cylinder deformation of $C'$ and $D$; more precisely it increases the height of $C'$ and decreases the height of $D$. Since both $C'$ and $D$ are in the equivalence class of the surviving cylinders in $\cC$, and since the marked point is now not contained in the closure of this equivalence class, this contradicts the initial claim in Step 1. 
\end{proof}

\begin{sublem}\label{SL:MZ2}
The cylinder $D$ is simple and borders an equivalent cylinder along $B$.
\end{sublem}
\begin{proof}
Suppose first to a contradiction that $B$ consists of one multiplicity two saddle connection. By Proposition \ref{P:MZ} we have that $D$ must look like one of the cylinders in Figure~\ref{F:LeavingEC4}. 
\begin{figure}[h!]
\includegraphics[width=.9\linewidth]{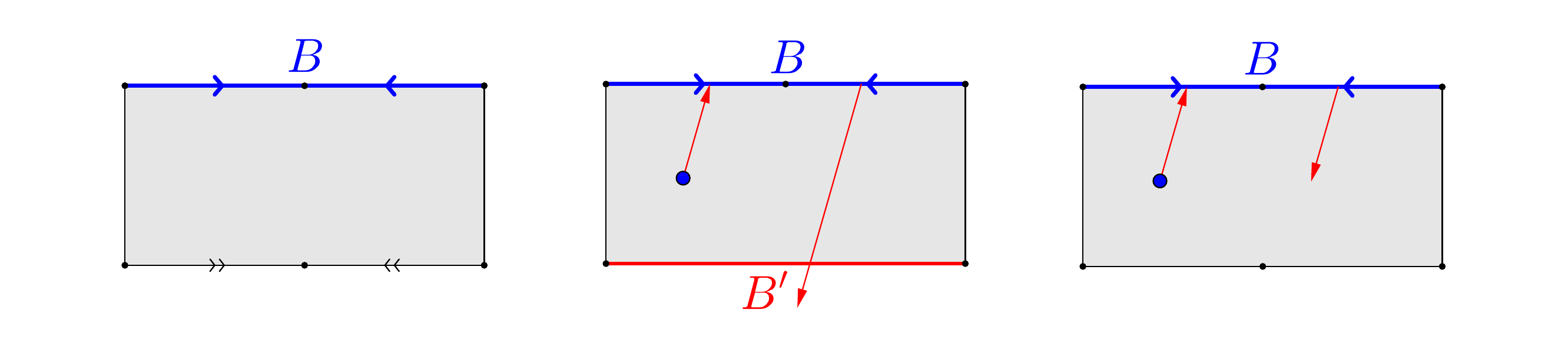}
\caption{Moving $p$ out of $\cC$ when $B$ contains one multiplicity two saddle connection. The cylinder shown is $D$.}
\label{F:LeavingEC4}
\end{figure}
The leftmost configuration only occurs in $\cQ(-1^4)$, which is rank one and hence precluded. In the middle configuration the boundary $B'$ opposite $B$ is a single multiplicity one saddle connection, which must lead out of $\cC$ by Proposition \ref{P:MZ}. Therefore, increasing the height of $C$ moves the marked point through $B$ and then out the opposite boundary of $D$ and hence out of $\cC$. In the rightmost configuration we increase the height of $C$ to move the marked point through $B$ and then into a situation where as the height of $C$ increases the marked point moves towards a boundary with two multiplicity one saddle connections. As in Sublemma~\ref{SL:MZ1}, specifically Figure \ref{F:LeavingEC1}, one of these saddle connections either leads out of $\cC$ or borders a cylinder other than $C$, which in turn leads out of $\cC$, so we reach a contradiction.

By Sublemma~\ref{SL:MZ1}, $B$ cannot consist of two multiplicity one saddle connections. Therefore, by Proposition \ref{P:MZ}, it must consist of a single multiplicity one saddle connection. If $B$ does not border another cylinder in $\cC$ then by increasing the height of $C$ we may move the marked point through $B$ and out of $\cC$, which is a contradiction (see the rightmost figure in Figure~\ref{F:LeavingEC1} for a similar case). Therefore, $B$ borders another cylinder  in $\cC$ and hence by Proposition \ref{P:MZ}, $D$ is simple. 
\end{proof}

Let $D'$ be the other cylinder in $\cC$ which borders $B$. Because $C$ is a simple cylinder or a simple envelope, $D'\neq C$. As $p$ moves into $D'$ it moves towards a boundary $B''$ of $D'$. We will reach a contradiction by moving $p$ out of $\cC$. Since all argument are very similar to those already given in this step, we will only sketch them here.

If $B''$ consists of one multiplicity two saddle connection then we proceed as in the middle left subfigure of Figure~\ref{F:FinalFig}. 
\begin{figure}[h!]
\includegraphics[width=\linewidth]{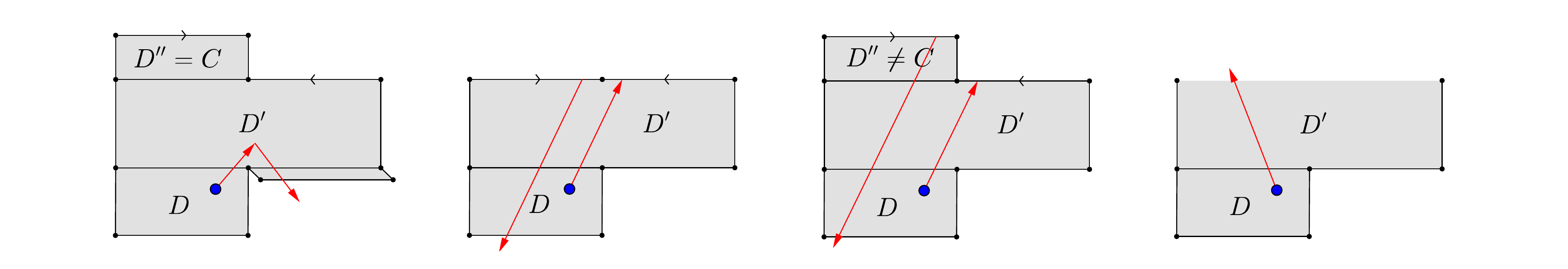}
\caption{Moving $p$ out of $\cC$ when $D$ is simple and borders an equivalent cylinder $D'$}
\label{F:FinalFig}
\end{figure}

If $B''$ borders a simple cylinder on both boundaries then call this cylinder $D''$. If $D'' = C$ then we may reduce the height of $C$ to move $p$ out of $\cC$ as in the leftmost subfigure of Figure~\ref{F:FinalFig} (if there is another cylinder bordering $D'$ then we first make it short so that we do not need to overcollapse $C$ to move $p$ out of $\cC$). Otherwise, we move $p$ out of $\cC$ as in middle right subfigure of Figure~\ref{F:FinalFig}. In all other cases we proceed as in the rightmost subfigure of Figure~\ref{F:FinalFig}.
\end{proof}

%
%

\section{Proof of Theorem \ref{T:findcyl}}\label{S:FindCyl}

To read this section it is necessary to be comfortable with the results of \cite{MZ} that are recalled in Proposition \ref{P:MZ}. Assumption \ref{A} is still in effect; otherwise $\cQ=\cH^{odd}(4)$ would be a counterexample to Theorem \ref{T:findcyl}. 


\begin{lem}\label{L:firstcyl}
Any $\cQ$ other than $\cQ(-1^4)$ contains a surface with a simple cylinder or two disjoint simple envelopes. 
\end{lem}

\begin{proof}
Pick disjoint cylinders $C, D$ on a  surface in $\cQ$ where all parallel saddle connections are hat homologous. If both are simple envelopes, or if either is a simple cylinder, we are done,  so we assume otherwise. In particular, assume $C$ is not a simple envelope or a simple cylinder.

By Proposition \ref{P:MZ}, since $C$ is not a simple envelope and is not a simple cylinder, it has exactly two distinct saddle  connections on at least one side, and cutting these saddle connections disconnects the surface into two components. Furthermore, the  component $R$ not containing $C$ has trivial holonomy. 

If $R$ contains a cylinder $C'$ hat homologous to $C$, then $C'$ must be simple. Otherwise, \cite{SW2} gives a cylinder $C'$ on $R$ not parallel to $C$, which must be simple.

(The use of \cite{SW2} can be avoided by an argument using square-tiled surfaces. The use of translation surface with boundary can be avoided by gluing together the two boundary saddle connections of $R$. We are using the fact that on a generic Abelian differential, every cylinder is simple.)
\end{proof}

\begin{lem}\label{L:secondcyl}
Suppose $(Q', q')$ is in a hyperelliptic component other than $\cQ(-1^4)$, $S$ is a set of non-singular points of the metric, and $c$ is a saddle connection on $(Q', q', S)$. Assume all points of $S$ are endpoints of $c$ (so $|S|\in \{0,1,2\}$). Then, possibly after moving $(Q', q', S)$ in its stratum in such a way that $c$ remains a saddle connection,  there is a simple cylinder $C'$ on $(Q', q')$ disjoint from $c$ and that does not contain any point of $S$ on its boundary. 
\end{lem}

\begin{proof} 
We assume all parallel saddle connections on $(Q', q')$ are hat homologous. Notice that since $\cQ$ is a hyperelliptic component  no cylinder is a simple envelope since every cylinder must be fixed by the hyperelliptic involution and a simple envelope does not admit an involution. 

First suppose $S$ is non-empty. Keeping in mind that the surface does not contain simple envelopes, Lemma \ref{L:firstcyl}  gives a cylinder $C'$ as desired, except that it may not be disjoint from $c$. However, moving the marked points we may make $c$ disjoint from $C'$. (Note that since $C'$ is simple it cannot cover the whole surface.)

So assume $S$ is empty. Let $C'$ be any cylinder disjoint from $c$. If $C'$ is simple we are done. Since the component is hyperelliptic, $C'$ can't be a simple envelope, and must have the same number of saddle connections on each of its two boundary components.  If $C'$ has two distinct saddle connections on each side then by Proposition \ref{P:MZ}, $C'$ disconnects the surface.  Cut the two saddle connections on the opposite side from $c$. By Proposition \ref{P:MZ}, the component $R$ not containing $C'$ or $c$ has trivial holonomy.  As in Lemma \ref{L:firstcyl}, we can find a simple cylinder in $R$. 

 The case that remains is that every cylinder $C'$ disjoint from $c$ is an envelope that has two multiplicity one saddle connections on one boundary. In this case  Proposition \ref{P:MZ} gives that the complement of $C'$ is connected and has trivial holonomy. In particular, there are only two poles, so we can't have two disjoint cylinders $C'$ of this type.

\begin{figure}[h!]
\includegraphics[width=.8\linewidth]{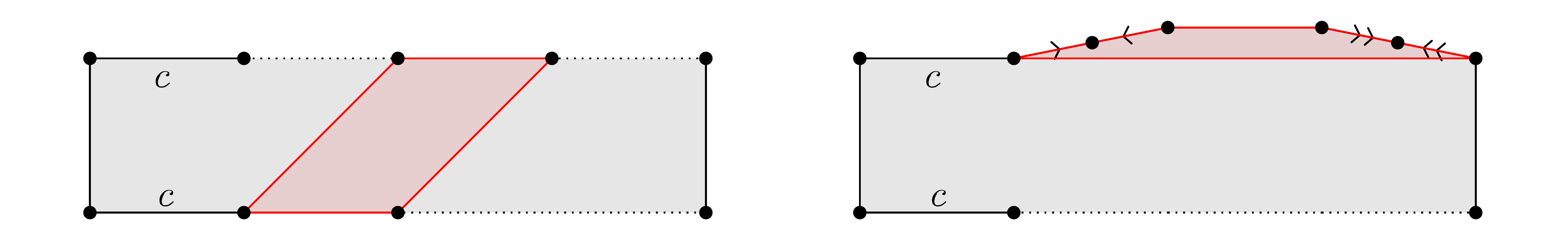}
\caption{The proof of Lemma \ref{L:secondcyl}.}
\label{F:Example}
\end{figure}

Assume $c$ is horizontal, and nudge the surface so that it becomes square-tiled. If there is more than one horizontal cylinder, by the previous comment at least one of them must not be as in the previous paragraph, so we  obtain a contradiction  (after nudging the surface again to restore the fact that  all parallel saddle connections are hat homologous). So assume there is just one horizontal cylinder. If a saddle connection other than $c$ appears both on the top and the bottom of this cylinder, then we can find a transverse simple cylinder disjoint from $c$, as in Figure \ref{F:Example} (left). So assume this is not the case. If two distinct saddle connections other than $c$ appear on the same side of the cylinder, we can nudge them to create a second horizontal cylinder while keeping the existing horizontal cylinder and $c$ horizontal, as illustrated in Figure \ref{F:Example} (right) in an example. (One can keep the surface horizontally periodic after the nudge by making it still be square-tiled, i.e. have rational period coordinates.) 

\begin{figure}[h!]
\includegraphics[width=\linewidth]{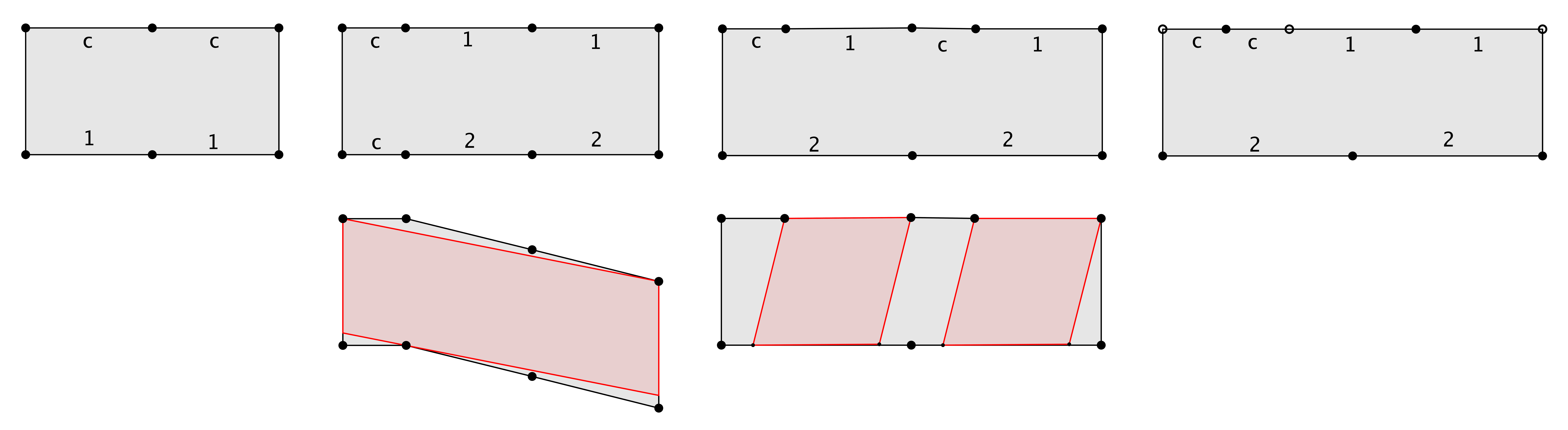}
\caption{The proof of Lemma \ref{L:secondcyl}.}
\label{F:BadStrebel}
\end{figure}
We are now in one of the four cases illustrated in the top of Figure \ref{F:BadStrebel}. Except that the left case is in $\cQ(-1^4)$ and hence excluded by our assumption that $(Q', q')\notin \cQ(-1^4)$, and the right case is excluded since there is a marked point.  For the middle two cases, the bottom of the figure shows how to find a simple cylinder disjoint from $c$. 
\end{proof}

\begin{lemma}\label{L:s-not-fixed}
Let $(Q', q')\in \cQ'$, and assume $\cQ'$ is hyperelliptic. Let $c$ be a saddle connection. Let $(Q, q)$ be the quadratic differential that arises from slitting $c$ and gluing in a simple cylinder. Then $(Q, q)$ belongs to a hyperelliptic component of a stratum of quadratic differentials if and only if $c$ is fixed by a hyperelliptic involution.
\end{lemma}

We omit the proof. By Assumption \ref{A} and Remark \ref{R:Unique}, the hyperelliptic involution is unique for $\cQ'$ except when $\cQ'=\cQ(-1^4)$. For $\cQ(-1^4)$, there are two hyperelliptic involutions.

\begin{lemma}\label{L:two-cylinders-no-hyp}
Suppose $C_1$ and $C_2$ are disjoint simple cylinders on  $(Q, q)$ such that collapsing either $C_i$ does not create marked points and gives a surface in a hyperelliptic component. Then $(Q, q)$ is in a hyperelliptic component or ${\mathcal{Q}}(5, -1)$. 
\end{lemma}

\begin{proof}
Degenerating both $C_i$ gives a hyperelliptic surface $(Q', q')$ with no marked points and with two saddle connections, $c_1$ and $c_2$. (If marked points were created, then even after moving the marked points  in an arbitrary way  there would have to be a hyperelliptic involution  preserving  the set of marked points. By assumption Assumption \ref{A}, $q'$ is not the square of a genus 1 Abelian differential. Hence $q'$ has at least one zero, and hence $(Q',q')$ has at mostly finitely many hyperelliptic involutions (typically only 1!). Hence it is is not possible that the set of marked points is preserved by a hyperelliptic involution even after being moved in an arbitrary way. This contrasts with $\cH(0,0)$.)

Since gluing in a cylinder to either $c_i$ gives a surface in a hyperelliptic component, Lemma \ref{L:s-not-fixed} gives that each $c_i$ is fixed by the involution. Hence  Lemma \ref{L:s-not-fixed} gives that $(Q,q)$ is hyperelliptic. 

\begin{figure}[h]
\includegraphics[width=.3\linewidth]{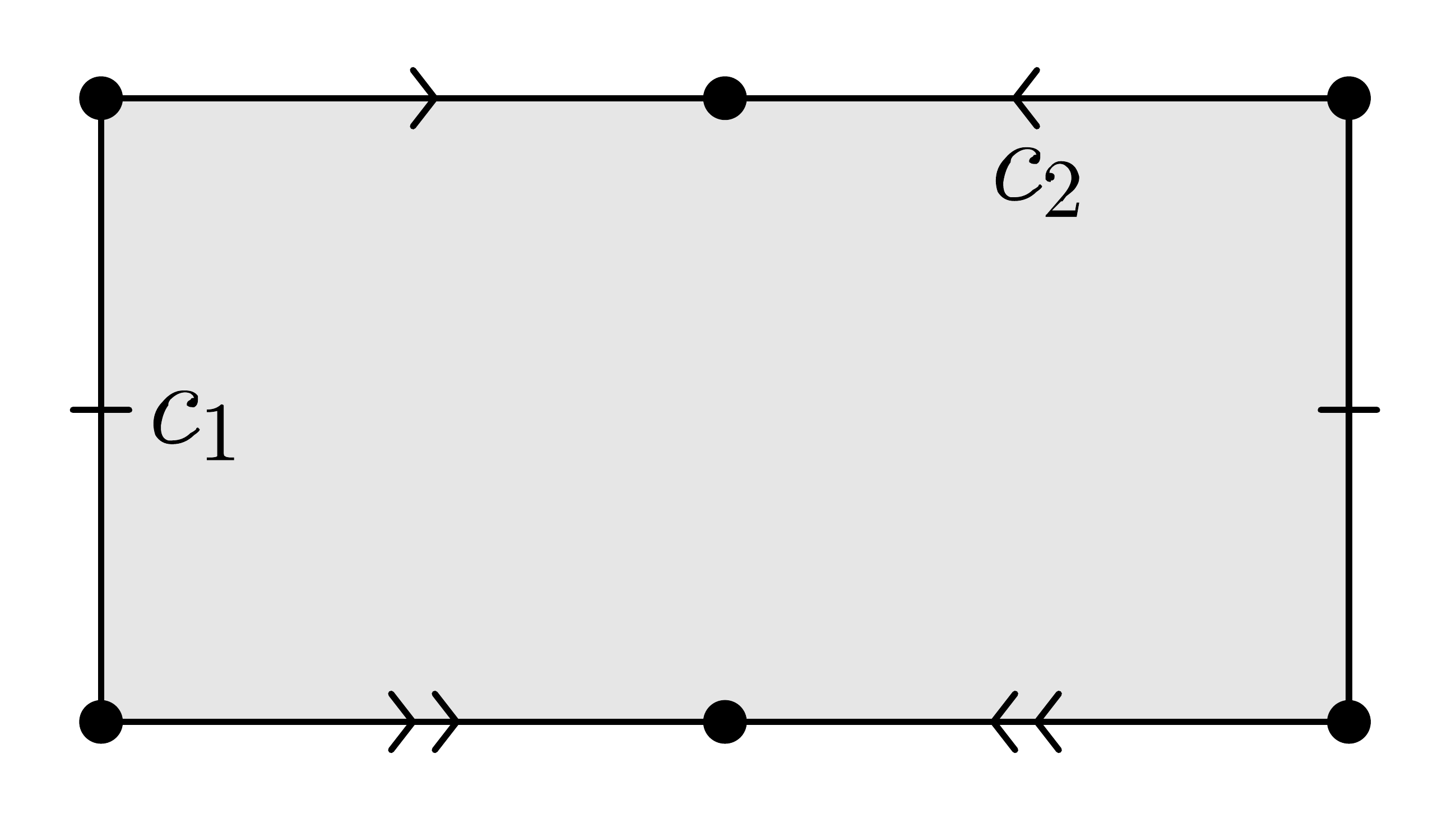}
\caption{If $c_1$ and $c_2$ are fixed by different hyperelliptic involutions, then up to $GL^+(2, \bR)$ the situation is as illustrated here. Compare to Figure  \ref{F:Prym3A}.}
\label{F:Pillow}
\end{figure}

This proof works as long as the hyperelliptic involution on $(Q', q')$ is unique, which is true for all strata but $\cQ(-1^4)$ and $\cH(\emptyset)$, the latter of which cannot arise here. If $(Q', q')\in \cQ(-1^4)$, there is the possibility that $c_1$ is fixed by one hyperelliptic involution, and $c_2$ by another, as in Figure \ref{F:Pillow}. In this case  $(Q, q)\in \mathcal{Q}(5, -1)$. 
\end{proof}

\begin{proof}[Proof of Theorem \ref{T:findcyl}]
Let $\cQ$ be non-hyperelliptic and not $\mathcal{Q}(3, -1^3),$ $\mathcal{Q}(5, -1)$, or $\mathcal{Q}(1, -1^5)$. 

Let $C_1$ be a cylinder that is either simple or a simple envelope on some $(Q, q)\in \cQ$. This exists by Lemma \ref{L:firstcyl}. If possible, pick $C_1$ so that degenerating it produces marked points; otherwise we will assume that degenerating any simple cylinder or simple envelope does not produce marked points.

Let $(Q_1, q_1, S_1)$ be the result of degenerating $C_1$. The cylinder $C_1$ becomes a saddle connection $c_1$ on $(Q_1, q_1, S_1)$. We may assume $(Q_1, q_1)$ is in a hyperelliptic component. 

\bold{Case 1: $(Q_1, q_1)\notin \cQ(-1^4)$.} Let $C_2$ be the cylinder on  $(Q_1, q_1, S_1)$ given by Lemma \ref{L:secondcyl}, which is disjoint from $c_1$. There is a corresponding cylinder on $(Q,q)$, which we also call $C_2$. Let $(Q_2, q_2, S_2)$ be the result of degenerating $C_2$ on $(Q,q)$. We may assume $(Q_2, q_2)$ is in a hyperelliptic component. Since surfaces in hyperelliptic components don't contain simple envelopes, we get that both $C_i$ are simple. 

We now claim that $S_1$ is empty. Otherwise, since $C_2$ does not contain any points of $S_1$ in its boundary, we see that degenerating $C_1$ on $(Q_2, q_2)$ (with marked points $S_2$ forgotten) produces a surface with marked points. Since $(Q_2, q_2)$ is contained in a hyperelliptic component, every element of the resulting stratum of surfaces with marked points must have an involution fixing the set of marked points. This implies the resulting stratum is $\cH(0,0)$, which contradicts Assumption \ref{A}. 

Now assume $S_1$ is empty. Our choice of $C_1$ implies that $S_2$ is empty. Hence Lemma \ref{L:two-cylinders-no-hyp} contradicts the assumption that $\cQ$ is non-hyperelliptic. 

\begin{figure}[h!]
\includegraphics[width=.2\linewidth]{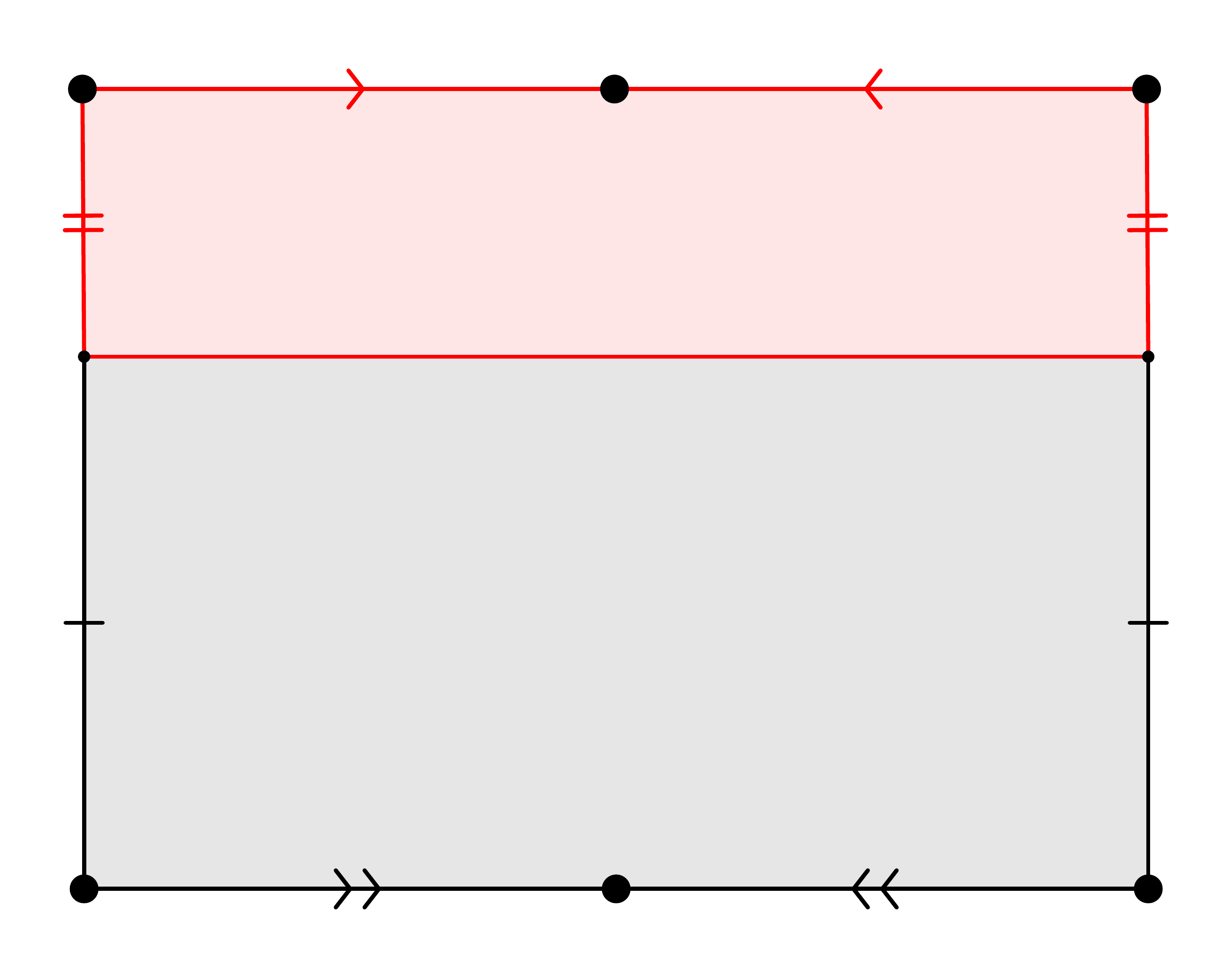}
\caption{Gluing a simple envelope onto a pillowcase with no marked points does not change the stratum.}
\label{F:NoChange}
\end{figure}

\bold{Case 2: $(Q_1, q_1)\in \cQ(-1^4)$.} 
If $S_1=\emptyset$, as in Figure \ref{F:NoChange} we get that $C_1$ is a simple cylinder,  since otherwise $(Q,q)\in \cQ(-1^4,0)$, a contradiction. As in Figure \ref{F:SmallCases} (top left) we conclude   that $\cQ=\cQ(2,-1,-1)$. This contradicts our assumption that $\cQ$ is not hyperelliptic. 

\begin{figure}[h!]
\includegraphics[width=.65\linewidth]{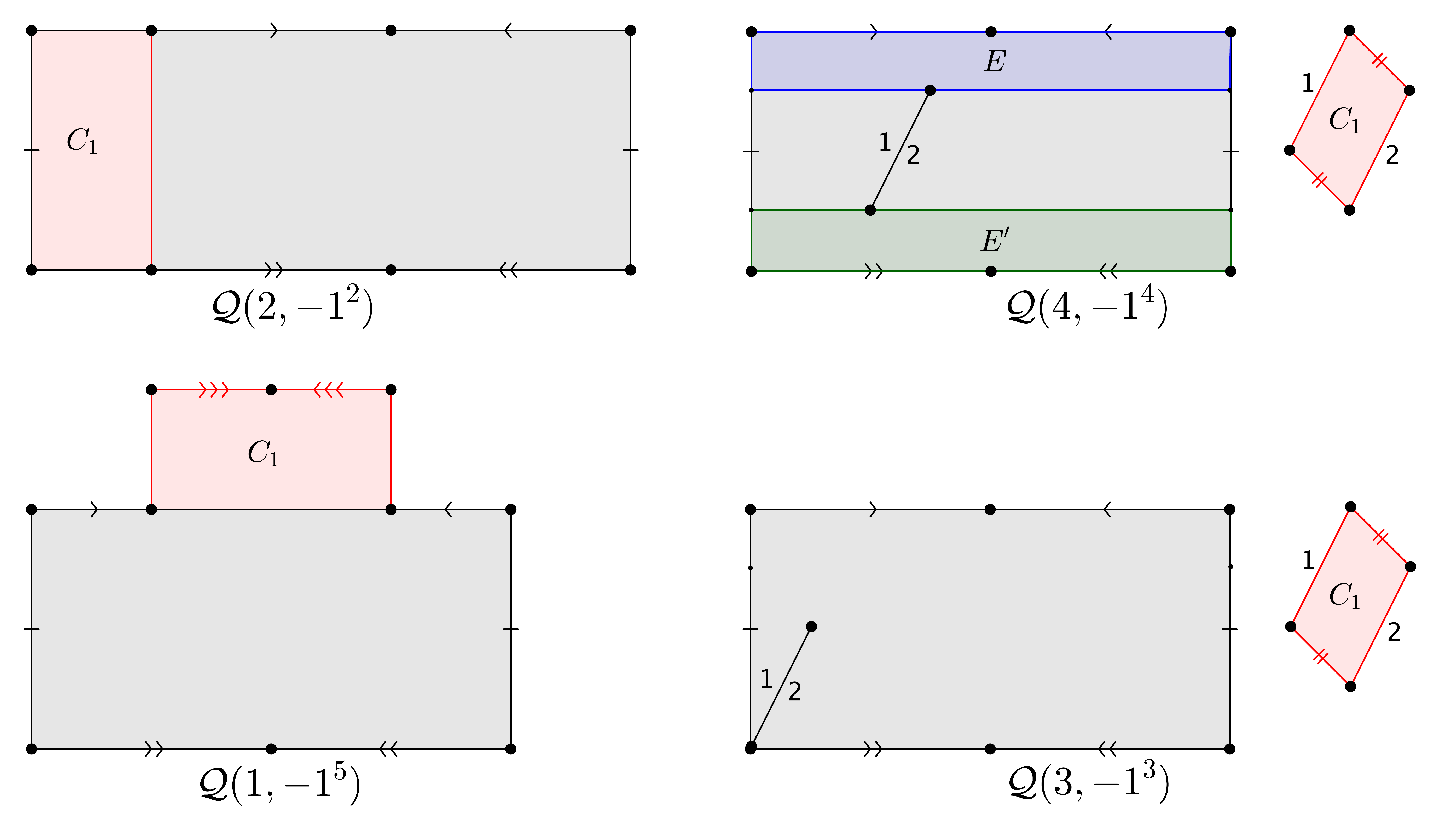}
\caption{The proof of Theorem \ref{T:findcyl}.}
\label{F:SmallCases}
\end{figure}

If $|S_1|=2$, then $C_1$ was simple and $\cQ=\cQ(4, -1^4)$. See Figure \ref{F:SmallCases} (top right). One can find a pair of disjoint simple envelopes $E, E'$ as in the figure such that degenerating either gives a surface in $\cQ(3, -1^3)$. 

If $|S_1|=1$, and $C_1$ was a simple envelope, then $\cQ=\cQ(1, -1^5)$. See Figure \ref{F:SmallCases} (bottom left).

If $|S_1|$=1 and $C_1$ was simple, $\cQ= \cQ(3, -1^3)$. See Figure \ref{F:SmallCases} (bottom right).
\end{proof}

%
%

\section{Proof of Theorem \ref{T:basecase}}\label{S:BaseCase}

The case of $\mathcal{Q}(1, -1^5)$ follows from \cite{Apisa}, which in particular classifies $\cH(2)=\tilde{\mathcal{Q}}(1, -1^5)$-periodic points. So we need only treat $\mathcal{Q}(3, -1^3)$ and $\mathcal{Q}(5, -1)$.


\begin{lemma}[Lemma 5.5 \cite{Apisa}]\label{calc1}
Let $(X, \omega)$ be a generic translation surface in an affine invariant submanifold $\M$. Let $C$ be a horizontal cylinder, and let $\cD_1, \cD_2$ be two vertical distinct $\M$-equivalence classes of cylinders such that the following two conditions hold:
\begin{enumerate}
\item The intersection of $\overline{\cup \cD_i}$ with the interior of $C$ is connected and nonempty for $i = 1, 2$. (Here $\cup \cD_1$ denotes the union of the cylinders in $\cD_1$, etc.)
\item Any cylinder equivalent to $C$ has a modulus that is an integer multiple of the modulus of $C$. 
\end{enumerate}
If $p$ is an $\M$-periodic point in the interior of $C$, then up to relabelling $\cD_1$ and $\cD_2$, the point $p$ is at the center of the rectangle given by the intersection of $\overline{\cup \cD_1}$  and $C$. Furthermore, removing  $\overline{\cup \cD_1}$ and $\overline{\cup \cD_2}$  divides $C$ into two rectangles of equal size. 

\begin{figure}[H]
            \centering
        \resizebox{.6\linewidth}{!}{\begin{tikzpicture}
                \draw[dashed] (0,3) -- (0,2);
                \draw (0,2) -- (0,0) -- (7,0);
                \draw[dashed] (1,3)--(1,0);
                \draw (1,2) -- (4,2);
                \draw[dashed] (4,3) -- (4,0);
                \draw[dashed] (5,3) -- (5,0);
                \draw (5,2) -- (7,2);
                \draw[black, fill] (.5, 1) circle[radius = 1.6pt];
                \node at (.5, .75) {$p$};
                \draw (-1,0) -- (-1,1) -- (-1.25, 1) -- (-.75, 1) -- (-1, 1) -- (-1,0) -- (-1.25, 0) -- (-.75, 0); \node at (-1.25, .5) {$h$};
                \draw (-2,0) -- (-2,2) -- (-2.25, 2) -- (-1.75, 2) -- (-2, 2) -- (-2,0) -- (-2.25, 0) -- (-1.75, 0); \node at (-2.25, 1) {$1$};
		\node at (.5, 3) {$\cD_1$}; \node at (4.5, 3) {$\cD_2$}; \node at (7.5, 1) {$C$};
		\draw (0, -1) -- (.5, -1) -- (.5, -.75) -- (.5, -1.25) -- (.5, -1) -- (0, -1) -- (0, -1.25) -- (0, -.75); \node at (.25, -1.25) {$k \ell_1$};
		\draw (0, -2) -- (1, -2) -- (1, -1.75) -- (1, -2.25) -- (1, -2) -- (0, -2) -- (0, -1.75) -- (0, -2.25); \node at (.5, -2.25) {$ \ell_1$};
		\draw (1, -1) -- (4, -1) -- (4, -.75) -- (4, -1.25) -- (4, -1) -- (1, -1) -- (1, -1.25) -- (1, -.75); \node at (2.5, -1.25) {$a$};
		\draw (4, -2) -- (5, -2) -- (5, -1.75) -- (5, -2.25) -- (5, -2) -- (4, -2) -- (4, -1.75) -- (4, -2.25); \node at (4.5, -2.25) {$ \ell_2$};
		\draw (5, -1) -- (8, -1) -- (8, -.75) -- (8, -1.25) -- (8, -1) -- (5, -1) -- (5, -1.25) -- (5, -.75); \node at (6.5, -1.25) {$b$};
            \end{tikzpicture}}
             \caption{Lemma \ref{calc1} asserts that, after scaling so $C$ has height 1, we have $a = b$ and $k = h = \frac{1}{2}$.}
            \label{fig:shearing}
\end{figure}
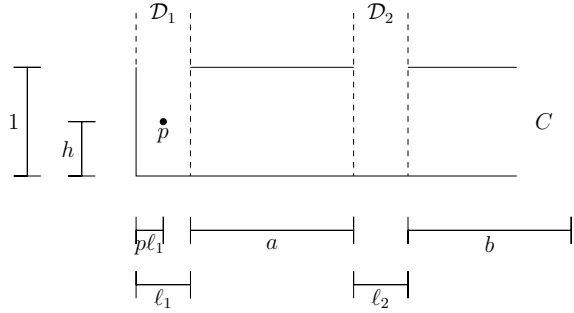
\end{lemma}

The lemma allows for multiple intersections of cylinders in $\cD_i$ and $C$, but requires that the union of these intersections be a single rectangle for each $i$.

We will apply Lemma~\ref{calc1} to the  vertically and horizontally periodic surfaces in Figure \ref{F:Prym3A}. Note that our convention is that, except where indicated otherwise,  opposite edges are identified when giving polygonal presentations for surfaces.  

The Cylinder Deformation Theorem implies that these surfaces are generic whenever the ratio of moduli of two inequivalent vertical and horizontal cylinders is irrational. Therefore, identifying the periodic points on these generic surfaces is sufficient to identify the periodic points in the corresponding strata of quadratic differentials.   

\begin{figure}[H]
    \begin{subfigure}[b]{0.31\textwidth}
        \centering
        \resizebox{.8\linewidth}{!}{\begin{tikzpicture}
        		\draw (0,0) -- (0,1) -- (1,1) -- (1,3) -- (3,3) -- (3,2) -- (2,2) -- (2,0) -- (0,0);
		\draw[dotted] (1,0) -- (1,1) -- (2,1);
		\draw[dotted] (1,2) -- (2,2) -- (2,3);
		\draw[black, fill] (.5,.5) circle[radius=1pt];
		\draw[black, fill] (1.5,.5) circle[radius=1pt];
		\draw[black, fill] (1.5,2.5) circle[radius=1pt];
		\draw[black, fill] (2.5,2.5) circle[radius=1pt];
                \draw[black] (1.5,1.5) circle[radius=2pt];
		\draw[black] (1.5,3) circle[radius=2pt];        
	\end{tikzpicture}
        }
        \label{SF:Prym3}
    \end{subfigure}
    \qquad
        \begin{subfigure}[b]{0.4\textwidth}
        \centering
        \resizebox{.8\linewidth}{!}{\begin{tikzpicture}
        		\draw (0,0) -- (0,2) -- (-1,2) -- (-1,3) -- (1,3) -- (1,4) -- (2,4) -- (2,2) -- (3,2) -- (3,1) -- (1,1) --  (1,0) -- (0,0);
		\draw[dotted] (0,1) -- (1,1) -- (1,3) -- (2,3) -- (2,2) -- (0,2) -- (0,3) -- (2,3) -- (2,1);
		\draw[black, fill] (-.5,2.5) circle[radius=1pt];
		\draw[black, fill] (1,2.5) circle[radius=1pt];
		\draw[black, fill] (1,1.5) circle[radius=1pt];
		\draw[black, fill] (2.5,1.5) circle[radius=1pt];
		%
		\draw[black] (.5,.5) circle[radius=2pt];
		\draw[black] (.5,2) circle[radius=2pt];
		\draw[black] (1.5,2) circle[radius=2pt];
		\draw[black] (1.5,3.5) circle[radius=2pt];  

	\end{tikzpicture}
        }
        \label{SF:Prym4}
    \end{subfigure}
\caption{$\tilde{\mathcal{Q}}(3, -1^3)$ (left) and $\tilde{\mathcal{Q}}(5, -1)$ (right).} 
\label{F:Prym3A}
\end{figure}
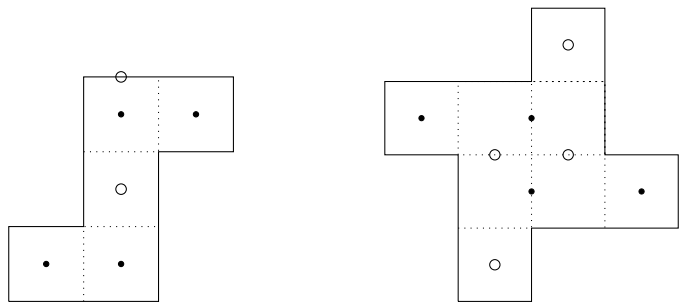 

\noindent \textbf{Periodic points in $\mathcal{Q}(3, -1^3)$: } Consider the surface on the left in Figure~\ref{F:Prym3A}. Letting $C$ be either the top or bottom horizontal cylinder, and $\cD_1$ and $\cD_2$  be the two equivalence classes of vertical cylinders, Lemma~\ref{calc1} implies that any periodic point contained in one of these two horizontal cylinders must be at the points indicated with solid dots. Letting $C$ be the middle vertical cylinder and $\cD_1$ and $\cD_2$ be the two equivalence classes of horizontal cylinders, Lemma~\ref{calc1} implies that any periodic point in $C$ must at the points indicated with circles. Since any of the solid dots can be moved into the middle vertical cylinder by a Dehn twist in horizontal cylinders, the solid dots are not periodic points (since they don't move onto circled points). The Dehn twist is possible by the Cylinder Deformation Theorem. 

Note the three cylinders labelled $C$ in the preceding paragraph cover the whole surface except for vertical saddle connection in the middle vertical cylinder and two horizontal saddle connections on the top and bottom horizontal cylinder. However, a point on these saddle connections can be moved off it by a Dehn twist in the simple horizontal or vertical cylinder whose core curve crosses the saddle connection. We conclude that any periodic point must lie in the orbit of the points marked with circles and hence be a fixed point of the involution. This proves Theorem~\ref{T:basecase} for $\tilde{\mathcal{Q}}(3, -1^3)$. 

\noindent \textbf{Periodic points in $\mathcal{Q}(5, -1)$: } Consider the surface on the right in Figure \ref{F:Prym3A}. Similarly to the previous case, setting $C$ to be either of the two middle horizontal cylinders and the $\cD_i$ to be the two vertical equivalence classes, Lemma~\ref{calc1} implies that any periodic point in the union of the two middle horizontal cylinders must be one of the solid dots. Similarly, setting $C$ to be either of the two middle vertical cylinders and the $\cD_i$ to be the two horizontal equivalence classes, Lemma~\ref{calc1} implies that that any periodic point in the union of the two middle vertical cylinders must be one of the circled points.

The central point of the surface is fixed by the involution and is hence a periodic point; let us exclude this from our discussion. Any other point can be moved to the interior of one of the four cylinders labelled $C$ in the previous paragraph using Dehn twists. To conclude our analysis of $\tilde{\mathcal{Q}}(5, -1)$ it suffices to show that none of the eight solid or circled points drawn on the surface on the right in Figure \ref{F:Prym3A}  are periodic. 

By using Dehn twists and symmetry, it suffices to show that the point $p$ in Figure \ref{F:Prym4Special} is not periodic. 

\begin{figure}[h]
            \centering
        \resizebox{0.45\linewidth}{!}{\begin{tikzpicture}
        		\draw (0,0) -- (0,3) -- (-1,3) -- (-1,5) -- (2,5) -- (2,6) -- (4,6) -- (4,3) -- (5,3) -- (5,1) -- (2,1) -- (2,0) -- (0,0);
		\node at (-.25, .5) {$a$};
		\node at (-.25, 2) {$1$};
		\node at (-1.25, 4) {$1$};
		\node at (-.5, 5.5) {$a$};
		\node at (1, 5.5) {$1$};
		\node at (2.5, 6.5) {$a$};
		\node at (3.5, 6.5) {$1-a$};
		\draw[dashed] (-1,4) -- (0, 5);
		\draw[dashed] (-1,3) -- (1,5);
		\draw[dashed] (0,3) -- (3,6);
		\draw[dashed] (0,2) -- (4,6);
		\draw[dashed] (0,1) -- (4,5);
		\draw[dashed] (0,0) -- (4,4);
		\draw[dashed] (1,0) -- (4,3);
		\draw[dashed] (3,1) -- (5,3);
		\draw[dashed] (4,1) -- (5,2);
		\draw[black, fill] (2,3.9) circle[radius=1pt];
		\node at (2.25, 3.9) {$p$};
	\end{tikzpicture}}
\caption{Concluding the classification of periodic points in $\mathcal{Q}(5, -1)$} 
\label{F:Prym4Special}
\end{figure}

In this figure, for any $1>a>0$ the slope 1 direction decomposes into four cylinders. With $a=\frac12$, $p$ is on the boundary of one of the cylinders, but for nearby $a$ it is not. By continuity, after changing $a$ the point $p$ does not have rational height in the slope 1 cylinder in which it lies, showing that $p$ is not a periodic point by \cite[Lemma 5.4]{Apisa}. This proves Theorem~\ref{T:basecase} for $\tilde{\mathcal{Q}}(5, -1)$.

\bibliography{mybib}
\bibliographystyle{amsalpha}

\end{document}